\documentclass[a4paper,12pt]{article}
\usepackage{amsmath,amssymb}
\newtheorem{theorem}{Theorem}[section]
\newtheorem{lemma}[theorem]{Lemma}
\newtheorem{proposition}[theorem]{Proposition}
\newtheorem{corollary}[theorem]{Corollary}
\newenvironment{proof}{\normalsize {\sc Proof}:}{{\hfill $\Box$}}
\newenvironment{proofof}[1]{\normalsize {\sc Proof of #1}:}{{\hfill $\Box$}}
\def\margin_comment#1{\marginpar{\sffamily{\small #1\par}\normalfont}}
\def \N {\mathbb{N}}
\def \lexle{<_{\rm lex}}
\def \slexle{<_{\rm slex}}
\def \slexlee{\le_{\rm slex}}
\def \p#1{{\rm pre}[#1]}
\def \s#1{{\rm suf}[#1]}
\def \f#1{{\rm f}[#1]}
\def \l#1{{\rm l}[#1]}
\def \S {\mathcal{S}}
\def \DA{\mathrm{DA}}
\newenvironment{mylist}{\begin{list}{}{
\setlength{\parskip}{0mm}
\setlength{\topsep}{1mm}
\setlength{\parsep}{0mm}
\setlength{\itemsep}{0.5mm}
\setlength{\labelwidth}{7mm}
\setlength{\labelsep}{3mm}
\setlength{\itemindent}{0mm}
\setlength{\leftmargin}{12mm}
\setlength{\listparindent}{6mm}
}}{\end{list}}
\title{Artin groups of large type are shortlex automatic with regular geodesics}
\author{Derek Holt and Sarah Rees}
\date{30th March 2010}
\begin{document}
\maketitle

\begin{abstract}We prove that any Artin group of large type is shortlex automatic
with respect to its standard generating set, and that the set of all geodesic
words over the same generating set satisfies the Falsification by
Fellow-Traveller Property (FFTP) and hence is regular.\end{abstract}

\section{Introduction}

In this article we consider Artin groups of large type,
in their standard presentations.
The standard presentation for an Artin group over its
standard generating set $X =\{a_1,\ldots, a_n\}$ is as
\[ \langle a_1,\ldots ,a_n \mid {}_{m_{ij}}(a_i,a_j)=
{}_{m_{ji}}(a_j,a_i)\quad\hbox{\rm for each}\quad i \neq j \rangle, \]
where the integers $m_{ij}$ are the entries in a Coxeter matrix
(a symmetric $n \times n$ matrix $(m_{ij})$ with entries in
$\N \cup \{\infty\}$, $m_{ii}=1, m_{ij} \geq 2$, $\forall
i \neq j$), and where 
for generators $a,a'$ and $m \in \N$ we define ${}_m(a,a')$ to be the
word that is the product of $m$ alternating $a$'s and $a'$'s that
starts with $a$.
Adding the relations $a_i^2=1$ to those for the Artin group defines the associated Coxeter group,
which is more commonly presented as
\[ \langle a_1,\ldots ,a_n \mid (a_ia_j)^{m_{ij}}=1\quad\hbox{\rm for each}\quad i,j \rangle. \]

An Artin group is said to be of {\em spherical} or {\em finite} type if the associated
Coxeter group is finite, of {\em dihedral type} if the associated
Coxeter group is dihedral (or, equivalently, the standard generator set has two elements),
of {\em large type} if $m_{ij} \geq 3$ for all $i \neq j$,
and of {\em extra-large type} if $m_{ij} \geq 4$ for all $i \neq j$. 

The aim of this paper is to prove that Artin groups of large type are
shortlex automatic over the standard generating set $X$, for any ordering of
$A := X \cup X^{-1}$.
We shall show also that the set of all geodesic words over $A$ satisfies
the Falsification by Fellow-Traveller Property (FFTP)
(see~\cite{NeumannShapiro}), and hence is a regular set.
These two main results appear as Theorem~\ref{Wslred} and Theorem~\ref{fftp}.

We remind the reader that
a group $G=\langle X \rangle$ is defined to be shortlex automatic if the set
of minimal representatives in $G$ of words under the shortlex ordering,
with respect to some ordering of $X \cup X^{-1}$, is a regular language $L$,
and for some constant $k$, any two words
$w,v \in L$ with $|w^{-1}v|_G \leq 1$ `$k$-fellow travel'. Here we use $|u|_G$
to denote the word length of the minimal representative of $u$ in $G$;
words $w,v$ are defined to $k$-fellow travel if,
where $w(i),v(i)$ denote the prefixes of $w,v$ of length $i$, we have
$|w(i)^{-1}v(i)| \leq K$ for each $i=1,\ldots,\max\{|w|,|v|\}$.
An additional fellow traveller property could make the group biautomatic.
We do not attempt to give a complete introduction to this topic, but
refer the reader to \cite{ECHLPT} as a basic reference on automatic groups.

If, for an Artin group, $m_{ij} = \infty$ for all $i \ne j$, then the
group is free.
Since free groups are well understood and are known to be biautomatic,
we shall assume that this is not the case,
and define $M$ to be $2\max\{ m_{ij} \mid m_{ij} \ne \infty \}$.
This will be our fellow traveller constant for automaticity proofs.

It is known that Artin groups of spherical type  
\cite{Charney}, extra-large type \cite{Peifer},
large type with at most three generators \cite{BradyMcCammond}, or
right angled type \cite{HermillerMeier,VanWyk} are
biautomatic. The first two results were each proved by direct construction of
an appropriate regular language,
while the third result was proved
via the verification of appropriate small cancellation conditions on the groups.

Artin groups of spherical type are also known to be Garside,
and the language of geodesics in a Garside group with respect to the
Garside (rather than standard) generators was studied by Charney and Meier (\cite{CharneyMeier}).
The geodesics for 2-generator Artin groups
over the standard generating set were subsequently described by
Mairesse and Math\'eus in \cite{MairesseMatheus}.

The remainder of this paper is divided into three sections.
Section~\ref{dihedral_sec} discusses 2-generator Artin groups, the
structure of their geodesics, and the process of reduction
to them, and proves
Theorems~\ref{artin2} and \ref{2gen_fftp}. These are the
2-generator analogues of Theorems~\ref{Wslred} and \ref{fftp}, but hold 
for all 2-generator Artin groups, without requiring the groups to be of large
type; they are vital components of the higher rank results.
In the final two sections we consider Artin groups of large type.
Section~\ref{shortlex_sec} considers the process that
rewrites a word to shortlex normal form, and proves Theorem~\ref{Wslred},
while Section~\ref{geodesics_sec} is dedicated to the proof of
Theorem~\ref{fftp}. 

{\it Notational Conventions}: We use $a,b$, or $a_1,a_2,\ldots,a_n$ for
the fixed generators of an Artin group,
$X=\{a_1,\ldots,a_n\}$, $A=X\cup X^{-1}$.
We use the shortlex ordering $\slexle$ on $A^*$  relative to some
fixed but arbitrary ordering of $A$; $u \slexle v$ if either $u$ is shorter than $v$ or $u$ and $v$ have the same length but $u$ precedes $v$ lexicographically.
We call elements of $X$ {\em generators}, and elements of the larger set $A$
{\em letters};
a letter is {\it positive} if it is a generator, {\it negative} otherwise.
We define the {\em name} of the letters $a_i$ and $a_i^{-1}$ to be $a_i$.
We say that a word $w \in A^*$ {\em involves} the generator $a_i$ if
$w$ contains a letter with name $a_i$, and we call $w$ a
{\em 2-generator} word if it involves exactly two of the generators.
We shall generally use $x,y,z,t$ for generators in $X$ and
$g,h$ for letters in $A$.
Words in $A^*$ will be denoted by $u,v,w$ (possibly with subscripts) or
$\alpha,\beta,\gamma,\eta, \xi$. (Roughly speaking, the difference is that $u,v,w$
will be used for interesting subwords of a specified word, and the Greek
letters for subwords in which we are not interested.)
A {\it positive word} is one in $X^*$ and a {\it negative word} one in $(X^{-1})^*$; otherwise it is {\it unsigned}.
For $u,v \in A^*$, $u=v$ denotes equality as words, whereas $u=_Gv$ denotes
equality within the Artin group. The length of the word $w$ is denoted by $|w|$,
while as above $|w|_G$ denotes the length of a geodesic representative.

For any distinct letters $x$ and $y$ and a positive integer $r$, we define
alternating products ${}_r(x,y)$ and $(y,x)_r$. The product
${}_r(x,y)$, is defined, as it was earlier, to be the
word of length $r$ of alternating $x$ and $y$ starting
with $x$, while
$(y,x)_r$ is defined to be the word of length $r$ of alternating $x$ and $y$ ending
with $x$.
For example, $_6(x,y) = xyxyxy = (x,y)_6$, $_5(x,y) = xyxyx = 
(y,x)_5$.
We define both $_0(x,y)$ and $(y,x)_0$ to be the empty word.
For any nonempty word $w$, we define $\f{w} $ and $\l{w} $ to be
respectively the first and last letter of $w$, and $\p{w}$  and $\s{w}$ to be
the maximal proper prefix and suffix of $w$. So $w=\p{w} \l{w} =\f{w} \s{w} $.

\section{2-generator Artin groups}
\label{dihedral_sec}
The 2-generator subwords of words over the standard generators of an
Artin group of large type will play a significant role, so we first study
certain aspects of the 2-generator case.

Let \[ \DA_m = \langle a,b \mid {}_m(a,b) = {}_m(b,a) \rangle \] be a
2-generator (dihedral) Artin group with $m \ge 2$.
The element \[ \Delta:={}_m(a,b)=_{\DA_{m}}{}_m(b,a)\]
is called the {\em Garside element}.
If $m$ is even then $\Delta$ is central,
while if $m$ is odd then $a^{\Delta} = b$ and $\Delta^2$ is central.
Conjugation by $\Delta$ induces a permutation $\delta$ of order 2 or 1 on the
letters in $A$, and hence an automorphism $\delta$ of order 2 or 1 of the
free monoid $A^*$.

Let $w$ be a freely reduced word over $A=\{a,b,a^{-1},b^{-1}\}$.
Then we define $p(w)$ to be the minimum of $m$ and the length of
the longest subword of $w$ of alternating $a$'s and $b$'s (that is the
length of the longest subword of $w$ of the form $_r(a,b)$ or $_r(b,a)$).
Similarly, we define $n(w)$ to be the minimum of $m$ and the length of
the longest subword of $w$ of alternating $a^{-1}$'s and $b^{-1}$'s.
It is proved in \cite{MairesseMatheus} that $w$ is geodesic in $\DA_m$ if and only if
$p(w) +n(w) \leq m$.
If $p(w)+n(w) < m$, then $w$ is the unique geodesic representative of the group
element it defines, but if $p(w)+n(w)=m$ then there are other representatives.

For example, consider the case $m=3$ in which
\[ \DA_m= \langle a,b \mid aba=bab \rangle.\]
In this case $aba$ and $bab$ are two geodesic representatives of the
same element with $p(aba)=p(bab)=3,n(aba)=n(bab)=0$.
Less trivially, let $w=ab^2a^{-1}$. 
Then $p(w)=2$, $n(w)=1$, and so $w$ is geodesic. Since $b^{-1}\Delta =_{\DA_m} ab =_{\DA_m} \Delta a^{-1}$ and $\Delta b=_{DA_m} a\Delta$, we see that
\[ w =ab^2a^{-1}= _{\DA_m} b^{-1}\Delta ba^{-1} =_{\DA_m} b^{-1}a\Delta a^{-1} =_{\DA_m} b^{-1}a^2b \]
Based on what we have observed in these two pairs of
geodesic words, we shall identify a set of geodesic words which we shall call
{\em critical} words, and define an involution $\tau$ acting on that
set. The recognition of critical subwords of a word and their replacement by
their images under $\tau$ will turn out to be crucial to the recognition
of words in shortlex normal form, and to the rewriting of
words to that form, both for the dihedral Artin groups that we consider now and
for higher
rank Artin groups of large type.
Critical words $w$ in $\DA_m$
will be non-unique geodesic words (hence freely reduced with $p(w)+n(w)=m$).
From our definition we shall verify the following.

\begin{proposition}
\label{critical_properties}
For any critical word $w$:
\begin{mylist}
\item[(1)] $\tau(w)$ is also critical, it represents the same
element of $\DA_m$ as $w$, and $\tau(\tau(w))=w$.
\item[(2)] $p(\tau(w))=p(w)$ and $n(\tau(w)) = n(w)$.
\item[(3)] The names of the first letters of $w$ and $\tau(w)$ are
distinct, as are the names of the last letters of $w$ and $\tau(w)$.
\item[(4)] The first letters of $w$ and $\tau(w)$ have the same sign if $w$ is
positive or negative, but different signs if $w$ is unsigned; the same is true
of the last letters of $w$ and $\tau(w)$.
\item[(5)] $w$ and $\tau(w)$ $2m$-fellow travel.
\end{mylist}
Furthermore, any freely reduced word $w$ satisfying $p(w)+n(w)\geq m$ must 
contain at least one critical subword. 
\end{proposition}

A freely reduced, unsigned, geodesic word $w$ with $p(w)+n(w)=m$ is defined to
be {\em critical} if it is has either of the forms
\[ {}_p(x,y)\xi (z^{-1},t^{-1})_n \quad{\rm or}\quad
 {}_n(x^{-1},y^{-1})\xi (z,t)_p. \]
where $\{x,y\} = \{z,t\}=\{a,b\}$.
(Obviously these conditions put some restrictions on the subword $\xi $.)

We define a positive geodesic word $w$ to be critical if it has either
of the forms ${}_m(x,y) \xi$ or $\xi (x,y)_m$, and only the one positive alternating
subword of length $m$.
Similarly we define a negative geodesic word $w$ to be critical it is has either
of the forms ${}_m(x^{-1},y^{-1}) \xi$ or $\xi (x^{-1},y^{-1})_m$, and only
the one negative alternating
subword of length $m$.
In either case the uniqueness condition on the maximal alternating subword
ensures that a maximal alternating subword is either on the left side or the
right side but not both (unless $\xi$ is empty), and so the decomposition of
the word is uniquely defined.

The involution $\tau$ is defined in terms of the automorphism
$\delta$ of $A^*$ that we defined earlier.
Note that, for any word $w$, $\delta(w)$ is a word representing the
element $w^\Delta=_{\DA_m} \Delta^{-1}w\Delta=_{\DA_m} \Delta w \Delta^{-1}$.

For unsigned critical words, we define $\tau$ by
\begin{eqnarray*}
\tau({}_p(x,y)\xi (z^{-1},t^{-1})_n)
&:=& {}_n(y^{-1},x^{-1})\delta(\xi )(t,z)_p,\\
\tau({}_n(x^{-1},y^{-1})\xi (z,t)_p) 
&:=& {}_p(y,x)\delta(\xi )(t^{-1},z^{-1})_n.
\end{eqnarray*}

For positive and negative geodesic words, we define $\tau$ as
follows, where $\xi$ is non-empty in the final four equations.
\begin{eqnarray*}
\tau({}_m(x,y))&:=& {}_m(y,x),\\
\tau({}_m(x^{-1},y^{-1}))&:=& {}_m(y^{-1},x^{-1})\\
\tau({}_m(x,y) \xi) &:=& \delta(\xi) (z,t)_m,\quad\hbox{\rm where}\quad z=\l{\xi},\\ 
\tau(\xi (x,y)_m) &:=& {}_m(t,z)\delta(\xi),\quad\hbox{\rm where}\quad z=\f{\xi},\\ 
\tau({}_m(x^{-1},y^{-1}) \xi) &:=& \delta(\xi) (z^{-1},t^{-1})_m,\quad\hbox{\rm where}\quad z=\l{\xi}^{-1},\\ 
\tau(\xi (x^{-1},y^{-1})_m) &:=& {}_m(t^{-1},z^{-1})\delta(\xi),\quad\hbox{\rm where}\quad z=\f{\xi}^{-1}.
\end{eqnarray*}

\begin{proofof}{Proposition~\ref{critical_properties}}
Most of (1) is immediate from the definitions of critical words $w$, and of
their images under $\tau$.
To verify that $w$ and $\tau(w)$
represent the same group element, we observe that,
whenever $p+n=m$,
\[ {}_p(x,y)=_{\DA_m} {}_n(y^{-1},x^{-1})\Delta
\hbox{ and } \Delta(z^{-1},t^{-1})_n =_{\DA_m}(t,z)_p,\]
and so
\begin{eqnarray*}
 p(x,y)\xi (z^{-1},t^{-1})_n
&=_{\DA_m}& {}_n(y^{-1},x^{-1})\Delta \xi (z^{-1},t^{-1})_n\\
&=_{\DA_m}& 
{}_n(y^{-1},x^{-1})\delta(\xi )\Delta (z^{-1},t^{-1})_n\\
&=_{\DA_m}& {}_n(y^{-1},x^{-1})\delta(\xi )(t,z)_p.
\end{eqnarray*}
That $\tau(\tau(w))=w$  is clear for unsigned words $w$; for positive and negative
words it will follow from (3).

(2) is immediate from the definitions.

It is immediate from the definition that Property (3) holds
for an unsigned critical word.
A short calculation verifies that it also holds for
critical positive and negative words.
For example, for a critical positive word $w$ of the form ${}_m(x,y)$, 
the definition of $\tau$ clearly ensures that the
names of the last letters of $w$ and $\tau(w)$ are different.
If $\xi$ is non-empty, 
the fact that $w$ has a unique positive alternating subword of
length $m$ ensures, both when $m$ is odd and even, that
$\f{\xi}=\l{{}_m(x,y)}$, and so that $\f{\delta(\xi)} = y= \f{\tau(w)} \neq x=\f{w}$.

(4) is immediate from the descriptions of $w$
and $\tau(w)$. 

The fellow traveller property (5) follows from the observation that, 
for any prefix $\eta$ of $\xi$, we have
$\delta(\eta^{-1})\, {}_n(y^{-1},x^{-1})^{-1}\, {}_p(x,y)\eta =_{\DA_m}
\delta(\eta^{-1}) \Delta \eta =_{\DA_m} \Delta$, which has length at most $m$.
Note that the words ${}_(x,y)\eta$ and ${}_n(y^{-1},x^{-1})\delta(\eta)$ may not have the same length, but their length differs by $|p-m|\leq m$.
Hence the words fellow travel at distance at most $2m$.

Finally we observe that any word $w$ satisfying $p(w)+n(w) \geq m$ must
have a subword $w'$ with $p(w') + n(w')=m$. If $w'$ is unsigned, it must
either contain a subword
that begins with a positive alternating word of length $p(w')$ and
ends with a negative alternating word of length $n'(w)$ or 
contain a subword that begins with such
a negative alternating word and ends with such a positive alternating word.
Such a subword is critical. If  $w'$ is positive or negative, certainly any 
maximal alternating subword is critical. (There could also be other critical
subwords containing these.)
\end{proofof}

We define $T$ to be the set of all critical words.
We call $w$ {\em upper critical} if $\tau(w)\lexle w$ and {\em lower critical}
if $w \lexle \tau(w)$.
Note that Proposition~\ref{critical_properties} (3) and (4) ensure that
whether $w$ is upper or lower critical is determined by the first
letter of $w$ together with the fact of whether $w$ is positive,
negative or unsigned.

We easily deduce the following from Proposition~\ref{critical_properties}, which we record here since it is useful later on.
\begin{corollary}
\label{critical_subword}
Suppose that $w$ is critical. If $w_1$ is a prefix of $w$ that is also critical, then $\tau(w_1)$ begins with the same letter as $\tau(w)$. If $w_2$ is a suffix
of $w$ that is also critical, then $\tau(w_2)$ ends with the same letter as 
$\tau(w)$.
\end{corollary}

We already observed that any non-geodesic or even non-unique geodesic
must contain a critical subword. In fact we can use the critical
subwords within non-geodesics to reduce to geodesic form.

\begin{lemma}\label{geodlem}
Suppose that $w \in A^*$ is geodesic and $g \in A$.

If $wg$ is non-geodesic, then either $\l{w}=g^{-1}$
or $w$ has a critical suffix $v$ such that $\l{\tau(v)}=g^{-1}$.
Similarly, if $gw$ is non-geodesic, then either $\f{w}=g^{-1}$ or $w$ has
a critical prefix $v$ such that  $\f{\tau(v)}=g^{-1}$.
\end{lemma}
\begin{proof}
Let $p=p(w), n=n(w)$. Suppose that $wg$ is non-geodesic and that $w$ does not
end with $g^{-1}$, so $wg$ is freely reduced. Then $p(wg)+n(wg)>m$,
and since $w$ is geodesic,
we must have $p(w)+n(w)=m$ and $p(wg)+n(wg)=m+1$. If $g = z \in X$, then
$p(wg)=p+1$, and so $wg$ must end with an alternating positive
subword of length $p+1$. Then $wg$ (and hence $w$) also contains a
negative alternating subword of length $n$, and hence
$w$ has a critical suffix $v ={}_n(x^{-1},y^{-1})\xi(z,t)_p$
for which $\l{\tau(v)}=z^{-1}=g^{-1}$. (This is true even when $p=0$.)
Similarly, if $g=z^{-1}$ with $z \in X$ then $n(wg)=n+1$
and $w$ has a critical suffix $v ={}_p(x,y)\xi(z^{-1},t^{-1})_n$ with
$\l{\tau(v)} = z=g^{-1}$. 

We can deduce the second result by applying the first result to $w^{-1}$.
\end{proof}

In this article we are specifically interested in shortlex normal form.
We shall see that whenever $w$ is a freely reduced word that is not minimal
under the shortlex ordering then $w$ has a factorisation as $w_1w_2w_3$,
where $w_2$ is critical and
either $w_1\tau(w_2)w_3 \lexle w$ or $w_1\tau(w_2)w_3$ is not freely
reduced. In that case, we call the substitution of $\tau(w_2)$ for $w_2$ within $w$ together with any subsequent free reduction within $w_1\tau(w_2)w_3$ a
{\rm critical reduction} of $w$.

Where a critical reduction as above reduces $w$ lexicographically,
the first letter of $\tau(w_2)$ must precede the first letter of $w_2$
lexicographically. Where a critical reduction is length reducing there could be free cancellation at either end of $\tau(w_2)$; however we shall see that 
we can always select reductions in such a way that free cancellation is
at the right hand end of the critical subword. With this in mind we define
$W$ to be the set of freely reduced words that have
no factorisation as $w_1w_2w_3$ with $w_2$ critical that gives
either $\f{\tau(w_2)}\lexle \f{w_2}$ or
free cancellation between $\l{\tau(w_2)}$ and $\f{w_3}$.

\begin{theorem}\label{artin2}
The set $W$ is the set of shortlex minimal representatives for the 2-generator Artin group $\DA_m$.
\end{theorem}

\begin{proof}
Since both free and critical reductions to a word produce a word less than
it in the shortlex order, a shortlex minimal word must certainly be in $W$.

So now suppose that $w \in W$, but that $w$ is not shortlex minimal.
We may assume by induction that every subword of $w$ is shortlex minimal.

First suppose that $w$ is not geodesic.
Then, since $\p{w}$ is geodesic, Lemma~\ref{geodlem} implies that
$\p{w}$ has a critical suffix $w'$ such that $\l{\tau(w')} = \l{w}^{-1}$.
This contradicts $w \in W$.

So suppose that $w$ is geodesic but not shortlex minimal. Then $p+n=m$,
with $p=p(w), n=n(w)$. 
Let $v$ be the shortlex minimal representative of $w$.
Then, since every subword of $w$ is shortlex minimal, we
must have $\f{v} \lexle \f{w}$. Let $g=\f{v}$. 
Then $g^{-1}w$ represents the
same element as $\s{v}$, and hence is not geodesic. 
So by Lemma~\ref{geodlem}, $w$ has a critical prefix $w'$
with $\f{\tau(w')} = g$. But then $g \lexle \f{w}$ implies $\tau(w') \lexle w'$,
again contradicting $w \in W$.  
\end{proof}

This completes our proof of Theorem~\ref{artin2},  which is an essential
component of Theorem~\ref{Wslred}. We finish this section with some further
technical results on geodesics, which will be used in Section~\ref{geodesics_sec}.

\begin{lemma}
\label{2gen_geos}
Suppose that $w$ and $v$ are distinct geodesics in $\DA_m$
such that one can be obtained from the other by a single $\tau$-move,
and suppose that $\l{w}$ has name $a$.
Let $p=p(w)$, $n=n(w)$, and suppose that $p$ and $n$ are both non-zero.
Let $\sigma$ be the longest alternating suffix of $w$.
\begin{mylist}
\item[(1)] If $\sigma=(b,a)_p$, then $v$ has either $\sigma$ or
$(a^{-1},b^{-1})_n$ as a suffix.
\item[(2)] If $\sigma=(b^{-1},a^{-1})_n$, then $v$ has either $\sigma$ or
$(a,b)_p$ as a suffix.
\item[(3)] Otherwise $\sigma$ is also the longest alternating suffix of $v$.
\end{mylist}
\end{lemma}
\begin{proof}
In cases (1) and (2), there are critical suffices containing $\sigma$ and any
critical subword intersecting $\sigma$ must contain it. The result follows
immediately by looking at the effect of $\tau$ on such a subword.

In case (3), without loss of generality we may assume that $\sigma=(b,a)_k$,
with $k<p$, and
we may assume that $v$ is obtained from $w$ by applying a single $\tau$ move that
involves a critical subword $w'$ of $w$ immediately preceding $\sigma$; note
that $\sigma$ itself cannot intersect a critical subword. We suppose that $v$
contains a longer alternating suffix.
Then $\l{\tau(w')}$ must be whichever element
of $\{a,b\}$ is not the first letter of $\sigma$. But in
that case $\l{(w')}= \f{\sigma}^{-1}$, and hence $w$ is not freely
reduced, and cannot be geodesic. 
We have a contradiction, and so deduce that $\sigma$ is a longest alternating
suffix of $v$.
\end{proof}

\begin{corollary} \label{2gen_fftp}
Suppose that $w =_{\DA_m} v$ with $w,v$ both geodesic, and
$\l{w} \ne \l{v}$. Then a single $\tau$-move on a critical suffix of $w$
transforms $w$ to a geodesic word $v'$ that $2m$-fellow travels with $w$,
such that $v=_{\DA_m}v'$ and $\l{v'} = \l{v}$.
\end{corollary}
\begin{proof}
It follows immediately from Theorem~\ref{artin2}
that $w$ and $v$ are linked by a sequence of $\tau$-moves. 
Then $w$ and $v$ are either both positive, or both negative,
or by Lemma~\ref{2gen_geos}
one ends with a positive alternating word 
$(b,a)_p$ and the other with $(a^{-1},b^{-1})_n$,
where $p=p(w)$, $n=n(w)$.

When both words are positive, we may (without loss of generality) suppose that
$w$ has a minimal critical suffix $w'$ of the form
${}_m(a,b)\xi$ for some possibly empty word $\xi$.
We let $v'$ be the word derived from $w$ by applying a $\tau$-move to
$w'$. Then $v'$ $2m$-fellow travels with $w$, by Proposition~\ref{critical_properties}.
It follows from the definition of $\tau$ that $\tau(w')$ has 
its last letter distinct from $w'$, and hence this must be
the last symbol of $v$.
The argument is analogous when both words are negative.

So now we suppose that $p(w)$ and $n(w)$ are both non-zero.
Assuming that the name of $\l{w}$ is $a$ (and hence
the name of $\l{v}$ is $b$) we see that
$w$ has a critical suffix $w'$ that ends either with $(b,a)_p$,
or with $(b^{-1},a^{-1})_n$. 
Again we let $v'$ be the word derived from $w$ by applying a $\tau$-move to
$w'$.
Then $\tau(w')$ ends either with $b^{-1}$ or with $b$, and so
$\l{v'}$ has name $b$, the same as $v$.
\end{proof}

We can also deduce the following,
as is explained in Section~\ref{geodesics_sec} just before
Proposition~\ref{fftp}:
\begin{corollary}
For any $m$, the dihedral Artin group $\DA_m$ defined over its standard
generating set satisfies FFTP, and hence the set of all geodesics over that
generating set is regular.
\end{corollary}
Note that the regularity of this set of geodesics was already known,
\cite{MairesseMatheus}.

\begin{lemma}
\label{gju_lemma}
Suppose that for some letter $g$ and some $j\geq 1$, a $\tau$-move transforms
a geodesic word $g^ju$ in $\DA_m$ to a word $v$. Then there is a $\tau$-move
that transforms $gu$ to a word $v'$ with $\l{v'}=\l{v}$.
\end{lemma}
\begin{proof}
The given $\tau$-move transforms a critical subword $w$ of $g^ju$.
The result is immediate except when $w=g^{j'}u$ for some $j' \ge 1$.
It is clear from the definition of critical words that in this
case $gu'$ is also critical and that $\l{\tau(w)} = \l{\tau(gu')}$,
and the result follows.
\end{proof}

\section{Shortlex reduction in Artin groups of large type}
\label{shortlex_sec}

We assume from now on that $G=\langle X\rangle$ is an Artin group of large type
defined by a matrix $(m_{ij})$ with each $m_{ij} \ge 3$
and not all $m_{ij}$ infinite.

For any distinct pair of generators $a_i,a_j$, where $i<j$, we let $G(a_i,a_j)$ be the subgroup of $G$
generated by $a_i$ and $a_j$. It is clear that $G(a_i,a_j)$ is  a quotient of
the 2-generator Artin group $\DA_{m_{ij}}$, so that all equations between
words in the $\DA_{m_{ij}}$ also hold in $G(a_i,a_j)$; in fact it will follow from
Theorem~\ref{Wslred} that the two groups are isomorphic.

Now if $w$ is a 2-generator word in $a_i,a_j$,
we define $p(w)$ and $n(w)$ just as we did for words of $\DA_{m_{ij}}$
in Section~\ref{dihedral_sec}, we call $w$ critical if it satisfies the
definition of criticality of that section, and then we define $\tau(w)$ just 
as in that section. From Proposition ~\ref{critical_properties} we have $w =_G \tau(w)$.
We also define $\delta(\xi)$ for any subword $\xi$ of $w$, just as in
Section ~\ref{dihedral_sec}. We denote by $T_{ij}$ the set of critical words over $a_i,a_j$.

Of course we can define critical 2-generator words for any pair of
generators; we denote by $T$ the set of all such critical words (that is the
union of all $T_{ij}$).
The bijection $\tau$ from Section~\ref{dihedral_sec} is well defined on that set,
and the integer valued maps $p,n$ are well defined on the set of 2-generator
words.
We can also use the notation $\delta(\xi)$ without ambiguity,
for subwords $\xi$ of 2-generator words; even when $\xi$ itself involves only
one generator, it will always be clear which two generators are involved.

We shall say that a 2-generator word $w$ involving $a_i,a_j$ is
{\it 2-geodesic} if it is geodesic as a word in the
2-generator Artin group $\DA_{m_{ij}}$.
We know from the previous section that this is
the case if and only if $p(w) + n(w) \le m_{ij}$.
We do not know at this stage that such words are geodesics as elements of $G$,
but this will follow from Theorem~\ref{Wslred}.

Now suppose that $w$ is a freely reduced word over the Artin generators
and that $w= \alpha_1 u_1 \beta_1$ where $u_1 \in T_{i_1j_1}$ for
some $i_1,j_1$.
Then $\alpha_1 \tau(u_1)\beta_1$ may contain a critical
subword $u_2$ in a set $T_{i_2j_2}$ for which
$|\{i_1,j_1\} \cap \{i_2,j_2\}|=1$, where $u_2$ and $\tau(u_1)$
overlap in a single generator.
If $u_2$ overlaps the left hand end of $\tau(u_1)$ and, in
addition, the name of  $\l{\alpha_1}$ is not in $\{a_{i_1},a_{j_1}\}$
then we have a {\em critical left overlap}.
If $u_2$ overlaps the right hand end of $\tau(u_1)$ and, in
addition, the name of  $\f{\beta_1}$ is not in $\{a_{i_1},a_{j_1}\}$
then we have a {\em critical right overlap}.

We shall consider sequences
\begin{eqnarray*}&&\alpha_1u_1\beta_1,\\
\alpha_1\tau(u_1)\beta_1&=&\alpha_2u_2\beta_2,\\
\alpha_2 \tau(u_2) \beta_2&=& \alpha_3u_3\beta_3,\\&\ldots\\
\alpha_k \tau(u_k) \beta_k.
\end{eqnarray*}
of words that are all equal in the group, and
where either we have a critical left overlap at every step or a critical
right overlap at every step.

We call such a sequence a {\em leftward or rightward critical sequence}
of length $k$ for $w$.

For example, with $m_{12},m_{13},m_{23} = 3,4,5$ and writing $a,b,c$ for
$a_1,a_2,a_3$:
\begin{eqnarray*}
\alpha ca^2cab^{-1}c^{-1}b^2c(a^{-1}b^{-2}a) \beta,\\
\alpha ca^2ca(b^{-1}c^{-1}b^2cb)a^{-2}b^{-1} \beta,\\
\alpha (ca^2cac)bc^2b^{-1}c^{-1} a^{-2}b^{-1} \beta,\\
\alpha acac^2abc^2b^{-1}c^{-1} a^{-2}b^{-1} \beta
\end{eqnarray*}
is a leftward critical sequence of length 3 in which the words $u_1$, $u_2$, $u_3$ (defined
above) are
bracketed.

The following result, which we shall use in the proof of Theorem~\ref{Wslred},
is an easy consequence of Proposition~\ref{critical_properties}~(5).
We recall that $M = 2\max\{ m_{ij} \mid m_{ij} \ne \infty \}$.
\begin{lemma}\label{fellowtravel}
Suppose that $w'$ is derived from $w$ by the application of
a critical sequence. Then $w$ and $w'$ $M$-fellow travel.
\end{lemma}

We call a critical sequence a {\it reducing sequence} if 
$\alpha_k\tau(u_k) \beta_k$
is either not freely reduced or is less
than $\alpha_1 u_1\beta_1$ lexicographically,
and in the first case call it a
{\it length reducing sequence}, in the second a
{\it lex reducing sequence}.
In general, a reducing sequence of either type might be either
leftward or rightward, and a lex reducing sequence might be
either leftward or rightward; but 
in this article, we shall reduce words to shortlex normal
form using a combination of
rightward length reducing sequences that spark off free reductions at
the right hand ends of subwords $\tau(u_k)$,
and leftward lex reducing sequences for which $\tau(u_k) \lexle u_k$.

Now we define $W$ to be the set of all freely reduced words $w$ that
admit no rightward length reducing sequence
or leftward lex reducing of any length $k \ge 1$.
Note that this agrees with the definition of $W$ in the 2-generator case
in Section~\ref{dihedral_sec}. We call the words in $W$ {\em critically reduced}.

The following is the first of our two main results: 

\begin{theorem}\label{Wslred}
Let $G$ be an Artin group of large type, defined over its standard generating
set, and let $W$ be the set of words just defined. Then
$W$ is the set of shortlex minimal representatives of the elements of $G$,
and $G$ is shortlex automatic.
\end{theorem}

The complete proof contains a considerable amount of technical detail, which will be 
verified later in this section, as the proofs of three subsidiary results,
Propositions~\ref{Wclosedngen},  ~\ref{ac1} and \ref{ac2}.
But given those three propositions, the proof of the theorem itself is straightforward,
and so we give that now.

\begin{proofof}{Theorem~\ref{Wslred}}
The proof divides into two parts. First we show (1) that $W$ is the 
set of shortlex minimal representatives of the elements of $G$. Then (2)
we verify that $W$ is regular and satisfies the $M$-fellow traveller property.

We start our proof of (1) by defining a map $\rho: A^* \rightarrow W$; we shall verify
that application of $\rho$
reduces any word to shortlex minimal form.

First we define $\rho(w)=w$ for all $w \in W$. Note that $W$ is closed under subwords, and 
contains $\epsilon$, which is therefore fixed by $\rho$.

Now suppose that $w \in W$, and  that $g \in A$, but that $wg \not\in W$.
If $wg$ is not freely reduced, that the free reduction of $wg$ is a prefix
of $w$, and so is in $W$; we define $\rho(wg)$ to be that prefix.
Otherwise we can apply the following result (proof deferred):
\begin{proposition}
\label{Wclosedngen}
Suppose that $w \in W$ and $g \in A$ is such that $wg$ is freely reduced
but $wg \not \in W$. Then a single rightward length reducing or leftward
lex reducing sequence followed by a free reduction in the rightward case
can be applied to $wg$ to yield an element of $W$.
\end{proposition}
In the first case of the proposition, $wg$ admits a rightward length reducing sequence
followed by a free reduction to a representative of $wg$ within $W$, which we shall call $\rho_1(wg)$. 
In the second case,
$wg$ admits a leftward lex reducing sequence to an element of $W$, which we shall call $\rho_2(wg)$.
We define $\rho(wg)$ to be $\rho_1(wg)$ in the first case, and $\rho_2(wg)$
in the second case, assuming that the first case does not also occur.

In each of the three situations just considered it is clear that $\rho(wg)$
is an element of $W$ that
represents the same group element as $wg$, and that $\rho(wg) \slexle wg$.

We can now extend the definition of $\rho$ to the whole of $A^*$ using the 
recursive rule
$\rho(wg) = \rho(\rho(w)g)$ for $w \in W$, $g \in A$. Then
at most $|w|$ successive reductions reduce $w$ to the element $\rho(w)$ of $W$,
which we call the {\em reduction} of $w$.

We see that $\rho(w)=_G w$, that $\rho(w) \slexlee w$,
for any word $w$, and hence that the shortlex minimal representative
of any element is fixed by application of $\rho$ and so must be in $W$. To
prove (1) we need only to verify that every word in $W$
is shortlex minimal.

Now suppose that $w'$ is a word over $A^*$ that is not shortlex minimal, and $w$ is the shortlex representative
of the group element represented by $w'$. We can define a chain of words
$w_0=w',w_1,\ldots,w_k=w$, where, for each $i=0,\ldots,k-1$, $w_i$ is
transformed to $w_{i+1}$ either by the insertion or deletion of
a subword $gg^{-1}$, for some $g \in A$, or by the replacement of a subword
${}_m(a_i,a_j)$ by a subword ${}_m(a_j,a_i)$, for some $i \neq j$ and $m=m_{ij}$.
That $\rho(w_i)=\rho(w_{i+1})$ is guaranteed by the two results,
Proposition~\ref{ac1} and Proposition~\ref{ac2} (proofs deferred):
\begin{proposition}\label{ac1}
$\rho(wgg^{-1}) = w$, $\forall w \in W, g \in A$.
\end{proposition}
\begin{proposition}\label{ac2}
$\rho(w\,{}_{m_{ij}}(a_i,a_j)) = \rho(w\,{}_{m_{ij}}(a_j,a_i)),\,
\forall w \in W,\,1 \le i,j \le n$.
\end{proposition}
It follows that $\rho(w')=\rho(w)=w$, and so that $w' \not \in W$.
This completes the proof of (1).

Now it follows from the combination of Proposition~\ref{Wclosedngen}
and Lemma~\ref{fellowtravel} 
that $w$ and $\rho(wg)$ $M$-fellow travel for any $w \in W$, $g \in A$.
Hence we can describe $W$ as the set of words $w$ for which there
is no word $w'$ with $w'=_G w$ and $w' \slexle w$ that $M$-fellow travels with
$w$. Using this description of $W$ we can construct a finite state automaton to recognise it; hence $W$ is regular, and we have completed the proof of (2).
So $G$ is shortlex automatic.
\end{proofof}

The verification of the theorem will be complete once the three propositions
used in its proof have been verified. 
Before we embark on these proofs, we shall introduce some
more detailed notation for critical sequences and prove some technical
results about rightward length reducing and
leftward lex reducing sequences.

We start by considering rightward critical sequences.
If $w$ admits a rightward critical sequence, then 
$w= \alpha w_1\cdots w_k\beta$ where:
\begin{mylist}
\item[(i)] For $1 \le l \le k$, $w_l$ is a word over generators
$a_{i_l},a_{j_l}$
\item[(ii)] For each $1 \le l < k$,
$|\{i_l,j_l\} \cap \{i_{l+1},j_{l+1}\} |=1$,
the name of the final letter of $w_l$ is $a_i$ with
$i \not \in \{i_{l+1},j_{l+1}\}$, and the name of the first letter of
$w_{l+1}$ is $a_j$ with $j \not \in \{i_{l},j_{l}\}$.
\end{mylist}
We call $\alpha w_1\cdots w_k\beta$ a {\em rightward critical factorisation}
of $w$, with {\em factors} $w_1,w_2,\ldots,w_k$, and
{\em first term} $w_1$.

The chain of $\tau$-moves transforms $w$ through the sequence of words 
\begin{eqnarray*}
 w=\alpha w_1w_2\cdots w_k\beta&=& \alpha u_1w_2\cdots w_k \beta,\\
\alpha\tau(u_1)w_2\cdots w_k\beta,\\
\alpha\p{ \tau(u_1) } \tau ( \l{ \tau(u_1) } w_2) w_3\cdots w_k\beta &=&
\alpha u'_1\tau(u_2)w_3\cdots w_k\beta,\\
\alpha\p{\tau(u_1)} \p{ \tau (u_2) } \tau (\l{ \tau(u_2)} w_3)\cdots w_k\beta
&=&
\alpha u'_1u'_2\tau(u_3)\cdots w_k \beta,\\
\ldots,\\
\alpha\p{ \tau(u_1)} \p{\tau(u_2)}  \cdots \p{\tau(u_{k-1})}\tau(u_k)\beta &=& 
\alpha u'_1 u'_2 \cdots u'_{k-1}\tau(u_k)\beta, 
\end{eqnarray*}
where we define $u_1=w_1$, $u_l=\l{\tau(u_{l-1})} w_l$ for $1 < l \le k$,
and $u'_l=\p{\tau(u_l)}$ for $1 \le l \le k$.
We notice that $|w_1|=|u_1|=|u'_1|+1$, 
and for $l>1$ $|w_l|=|u'_l|=|u_l|-1$.

This sequence is length reducing when $\l{\tau(u_k)}=\f{\beta}^{-1}$,
and in this case we call the
letter $\f{\beta}$ the {\em tail} of the sequence.
Then the free reduction of the final word in the sequence
is \[\alpha u'_1 u'_2 u'_3 \cdots u'_k \s{\beta} \]
 
Figure \ref{fig1} illustrates a rightward length reducing sequence.

When a sequence of this type reduces a word of the form $wg$ with $w \in W$, 
then $\beta$ must be the single letter $g$, and then
the tail is $g$ too, and the whole of $\beta$.

\setlength{\unitlength}{0.48pt}
\begin{figure}
\begin{center}
{
\small
\begin{picture}(800,150)(-2,-10)
\put(50,40){$\alpha$}
\qbezier(0,50)(50,15)(100,50)
\qbezier(100,50)(175,120)(250,80) 
\put(180,100){$w_1$}
\put(180,80){$u_1$}
\qbezier(100,50)(130,25)(180,40) \qbezier(180,40)(200,50)(250,80)
\put(136,48){$\tau(u_1)$}
\put(140,15){$u'_1$}
\qbezier(250,80)(295,110)(370,80) \put(300,100){$w_2$}
\put(300,80){$u_2$}
\qbezier(232,68)(260,25)(300,40) \qbezier(300,40)(320,50)(370,80)
\put(256,48){$\tau(u_2)$}
\put(260,15){$u'_2$}
\qbezier(370,80)(415,110)(490,80) \put(420,100){$w_3$}
\put(420,80){$u_3$}
\qbezier(351,68)(380,25)(420,40) \qbezier(420,40)(440,50)(490,80) 
\put(376,48){$\tau(u_3)$}
\put(380,15){$u'_3$}
\qbezier(490,80)(535,110)(610,80) \put(540,100){$w_4$}
\put(540,80){$u_4$}
\qbezier(471,68)(500,25)(540,40) \qbezier(540,40)(560,50)(610,80) 
\put(498,48){$\tau(u_4)$}
\put(500,15){$u'_4$}
\qbezier(610,80)(655,110)(730,80) \put(660,100){$w_k$}
\put(660,80){$u_k$}
\qbezier(592,68)(620,25)(660,40) \qbezier(660,40)(680,50)(730,80) 
\put(618,48){$\tau(u_k)$}
\put(640,17){$u'_k$}
\put(703,53){$g^{-1}$}
\put(745,68){$g$}
\put(745,90){$\beta$}
\qbezier(730,80)(770,80)(810,120)
\end{picture}
}
\caption{\label{fig1} Rightward length reducing sequence for $w$, rewriting
$w=\alpha w_1\cdots w_k\beta$ as $\alpha u'_1\cdots u'_{k-1}\tau(u_k){\beta}$,
enabling free reduction of $\l{\tau(u_k)}=g^{-1}$ with $\f{\beta}=g$, and so
reduction of $w$ to $\alpha u'_1\cdots u'_k\s{\beta}$.}
\end{center}
\end{figure}
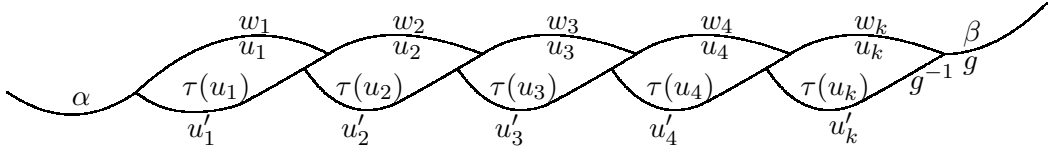

From now on,
whenever a word $w$ has a rightward critical factorisation
$\alpha w_1\cdots w_k\beta$ as above,
we will use the labels $u_1,\ldots,u_k$, $u'_1,\ldots,u'_k$ for subwords of
$w$ and its reductions through the rightward length reducing sequence
just as above. (And similarly, we shall define words
$\bar{u}_1,\ldots \bar{u}_{\bar{k}}$ and
$\bar{u}'_1,\ldots \bar{u}'_{\bar{k}}$ as labels for subwords associated
with a rightward critical factorisation 
of a word $\bar{\alpha}\bar{w}_1 \cdots \bar{w}_{\bar{k}}\bar{\beta}$.)

Now we consider leftward critical sequences.
If $w$ admits a leftward critical sequence then we can write
$w= \alpha w_k\cdots w_1\beta$ where:
\begin{mylist}
\item[(i)] For $1 \le l \le k$, $w_l$ is a word over generators
$a_{i_l},a_{j_l}$
\item[(ii)] For each $1 \le l < k$,
$|\{i_l,j_l\} \cap \{i_{l+1},j_{l+1}\} |=1$,
the name of the final letter of $w_{l+1}$ is $a_i$ with
$i \not \in \{i_l,j_l\}$, and the name of the first letter of
$w_{l}$ is $a_j$ with $j \not \in \{i_{l+1},j_{l+1}\}$.
\end{mylist}
We call $\alpha w_k\cdots w_1\beta$ a {\em leftward critical factorisation} 
of $w$, with {\em factors} $w_1,w_2,\ldots,w_k$, and {\em first term} $w_1$.

The chain of $\tau$ moves transforms $w$ through the sequence of words 
\begin{eqnarray*}
w=\alpha w_k\cdots w_2 w_1\beta &=& \alpha w_k\cdots w_2u_1\beta,\\
\alpha w_k \cdots w_2 \tau(u_1)\beta,\\
\alpha w_k \cdots w_3\tau(w_2\f{\tau(u_1)}) \s{\tau(u_1)}\beta &=&
\alpha w_k \cdots w_3\tau(u_2) u'_1\beta
,\\
\alpha w_k \cdots \tau(w_3\f{\tau(u_2)}) \s{\tau(u_2)} \s{\tau(u_1)}\beta &=&
\alpha w_k \cdots \tau(u_3) u'_2 u'_1\beta,\\
\ldots,\\
\alpha \tau(u_k)\s{\tau(u_{k-1})}\cdots \s{\tau(u_2)} \s{\tau(u_1)}\beta &=&
\alpha \tau(u_k)u'_{k-1}\cdots u'_2 u'_1\beta
\end{eqnarray*}
where we define $u_1=w_1$, $u_l=w_l\f{\tau(u_{l-1})} $ for $1 < l \le k$,
and $u'_l=\s{\tau(u_l)}$ for $1 \le l < k$. (We don't need to define $u_k'$
in this case.)
We notice that $|w_1|=|u_1|=|u'_1|+1$, 
and for $l>1$ $|w_l|=|u'_l|=|u_l|-1$.

The sequence is lex reducing when $\f{\tau(u_k)}$ is earlier
in the lexicographic order of generators than $\f{w_k}$,

Figure \ref{fig2} illustrates the leftward critical sequence.
\setlength{\unitlength}{0.48pt}
\begin{figure}
\begin{center}
{
\small
\begin{picture}(800,150)(-840,-10)
\put(-60,60){$\beta$}
\qbezier(-100,50)(-60,80)(-30,120)
\qbezier(-100,50)(-175,120)(-250,80) 
\put(-180,100){$w_1$}
\put(-180,76){$u_1$}
\qbezier(-100,50)(-130,25)(-180,40) \qbezier(-180,40)(-200,50)(-250,80)
\put(-185,46){$\tau(u_1)$}
\put(-180,15){$u'_1$}
\qbezier(-250,80)(-295,110)(-370,80) \put(-300,100){$w_2$}
\put(-300,80){$u_2$}
\qbezier(-231,68)(-260,25)(-300,40) \qbezier(-300,40)(-320,50)(-370,80)
\put(-305,48){$\tau(u_2)$}
\put(-300,15){$u'_2$}
\qbezier(-370,80)(-415,110)(-490,80) \put(-420,100){$w_3$}
\put(-420,80){$u_3$}
\qbezier(-351,68)(-380,25)(-420,40) \qbezier(-420,40)(-440,50)(-490,80) 
\put(-425,48){$\tau(u_3)$}
\put(-420,15){$u'_3$}
\qbezier(-490,80)(-535,110)(-610,80) \put(-540,100){$w_4$}
\put(-540,80){$u_4$}
\qbezier(-471,68)(-500,25)(-540,40) \qbezier(-540,40)(-560,50)(-610,80) 
\put(-545,48){$\tau(u_4)$}
\put(-540,15){$u'_4$}
\qbezier(-610,80)(-655,110)(-730,80) \put(-660,100){$w_k$}
\put(-660,80){$u_k$}
\qbezier(-591,68)(-620,25)(-660,40) \qbezier(-660,40)(-680,50)(-728,80) 
\put(-665,48){$\tau(u_k)$}
\put(-780,70){$\alpha$}
\qbezier(-730,80)(-780,45)(-830,80)
\end{picture}
}
\caption{\label{fig2} Leftward lex reducing sequence, reducing $w=\alpha w_k\cdots w_1\beta$ to $\alpha \tau(u_k)u'_{k-1}\cdots u'_1\beta$}
\end{center}
\end{figure}
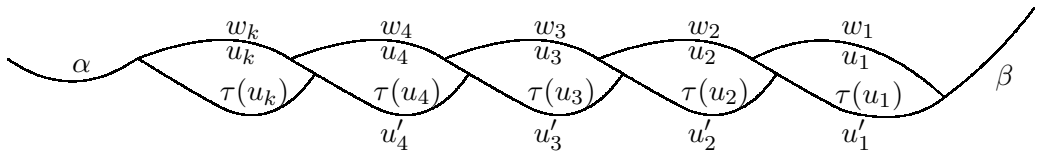

From now on,
whenever a word $w$ has a leftward critical factorisation
$\alpha w_k\cdots w_1\beta$ as above,
we will use the labels $u_1,\ldots,u_k$, $u'_1,\ldots,u'_{k-1}$ for subwords of
$w$ and its reductions through a leftward lex reducing sequence
as defined above. (And similarly, we shall define words
$\bar{u}_1,\ldots \bar{u}_{\bar{k}}$ and
$\bar{u}'_1,\ldots \bar{u}'_{\bar{k}-1}$ as labels for subwords associated
with a leftward critical factorisation of a word
$\bar{w} = \bar{\alpha}\bar{w}_{\bar{k}}\cdots \bar{w}_{1}\bar{\beta}$.)

Of course this notation is analogous to that used for rightward
critical factorisations, but with some differences; these should
not cause problems, since it will always be clear which type of
factorisation is being considered.

The following four technical results are used  in the proofs of
the three propositions,
Propositions~\ref{Wclosedngen},  ~\ref{ac1} and \ref{ac2}.

\begin{lemma}\label{rrslem}
Suppose that $wg$ admits a rightward length reducing sequence,
with corresponding factorisation $\alpha w_1\ldots w_kg$ of $wg$, and
notation as above.  Then the 
2-generator suffix $w_kg$ of $wg$ satisfies $p(w_kg)+n(w_kg) \ge m$,
and hence contains a critical subword.
\end{lemma}

\begin{proof} Since $\tau(u_k)g$ is not freely reduced, it is
not 2-geodesic and hence neither is $u_kg$. So
$p(u_kg) + n(u_kg)>m$, and hence 
$p(\s{u_k}g) + n(\s{u_k}g) \geq m$.
Since $\s{u_k}=w_k$ when $k>1$, while $w_1=u_1$,
the result now follows immediately.
\end{proof}

We call a rightward length reducing sequence for $wg$ {\em optimal}
if the left hand end of $w_1$ is further right in $w$ than in any other such factorisation.
We call a leftward lex reducing sequence for $wg$ {\em optimal}
if the left hand end of $w_k$ is further left in $w$ than in any other such factorisation.

\begin{lemma}\label{lem1}
Suppose that $wg$ admits an optimal rightward length reducing sequence,
with corresponding factorisation $\alpha w_1\cdots w_kg$ of $wg$, and
notation as above.
Then for each $l$ with $1 \le l \le k$:
\begin{mylist}
\item[(1)] No proper suffix of $u_l$ is critical;\\
hence
$u_l$ either has the form
${}_p(x,y)\xi (z^{-1},t^{-1})_n$  with $p>0$\\ or the form
${}_n(x^{-1},y^{-1})\xi (z,t)_p$ with $n>0$,
 and $\{x,y\}=\{z,t\} = \{a_{i_l},a_{j_l}\}$.
\item[(2)] $u_l'$ involves both of the generators $a_{i_l}$ and $a_{j_l}$.
\item[(3)] $p(u_l')+n(u_l') < m$.
\item[(4)] When $l>1$, $u_l'$  begins with a letter whose name is not in
$\{ a_{i_{l-1}}, a_{j_{l-1}}\}$.
\item[(5)] When $l<k$, each $u_l'$  ends with a letter whose name is not in
$\{ a_{i_{l+1}}, a_{j_{l+1}}\}$ and $u_k'$ ends with a letter 
with a different name from $g$.
\item[(6)] When $k>1$,
$w_2,\ldots, w_kg$ are maximal 2-generator subwords of $wg$,
and $u'_2,\ldots, u'_k$ are maximal 2-generator subwords of its reduction
$\alpha u'_1 \cdots u'_k$.
\item[(7)] If $\alpha u'_1\cdots u'_k$ admits a further left lex reducing
or right length reducing
sequence, then all of the factors of that sequence, as well as its tail when
length reducing, are contained within $\alpha u'_1$.
\end{mylist}
\end{lemma}

\begin{proof}
The fact (1) that no proper suffix of any $u_l$ is critical follows from
the optimality of the chosen sequence.  For if $u_0$ is a proper suffix of
$u_l$ that is critical, then $\tau(u_0)$, like $\tau(u_l)$, is critical,
and Corollary~\ref{critical_subword} tells us that $\tau(u_0)$ ends in the same
letter as $\tau(u_l)$, and hence also
has critical overlap with $w_{l+1}$. Since $u_0$ is also a suffix of $w_l$,
$\alpha'u_0w_{l+1}\cdots w_kg$ is the factorisation associated with
a rightward length reducing sequence for $wg$, where
$\alpha'= \alpha w_1\cdots w_{l-1}w_0$, for some prefix $w_0$ of $w_l$,
and the optimality of the chosen sequence is contradicted.

Once it is clear that $u_l$ has no critical suffix it is immediate that it
has one of the two given forms.
From now on we shall assume that it has the first form
${}_p(x,y)\xi (z^{-1},t^{-1})_n$  with $p>0$.

(2) is clear except possibly when $m = 3$ and $p=2,n=1$
with $u_l = xy\xi y^{-1}$. But in that case, $\xi$ is nonempty
and cannot start with $x$ or end with $x^{-1}$,
so $\xi$ must involve the generator $y$ and then
$\p{\tau(u_l)} $ involves both $x$ and $y$. So (2) holds.
(But note that (2) would not necessarily hold when $m = 2$, so
we are using the largeness assumption here.) 

If $p(\s{u_l})+n(\s{u_l})\geq m$ then either $\s{u_l}$ itself or some
suffix of it is critical (since $u_l$ is already critical),
and we have already excluded this possibility.
Hence $p(\s{u_l})+n(\s{u_l})<m$, so $p(\p{\tau(u_l)})+n(\p{\tau(u_l)})<m$
and (3) holds.

(4) follows immediately from the fact that the first letters of the
critical words
$u_l$ and $\tau(u_l)$ have different names.

Since $\p{\tau(u_l)}$ ends with $(z,t)_{p-1}$, we see that (5)
holds except possibly when $p=1$ and $n=m-1$. In that case, $p(\s{u_l} ) < p$
(which follows from (3)) implies that $\xi$ is either empty or a negative word.
If $\xi$ is empty, then $\p{\tau(u_l)}$ must end with $t^{-1}$ or else
$\tau(u_l)$ would not be freely reduced. Otherwise, the the last
letter of $\xi$ must be the same as the first letter of $(z^{-1},t^{-1})_n$
(since otherwise we would have a longer negative alternating word),
and hence, for both odd and even $m$, $\l{\delta(\xi)} = t^{-1}$, so
(5) holds in all cases.

(6) follows immediately from (4) and (5).

For (7) we may assume that $k>1$, or there is nothing to prove.
(3) implies that none of $u'_2,\ldots,u'_k$ can contain critical subwords.
Since they are maximal 2-generator subwords within $wg$ their concatenation
cannot contain or intersect any critical subword (where we have once again
used the largeness condition). 
Now the first term of any further reducing sequence is critical
so must be disjoint
from the suffix $u'_2\cdots u'_k$  of the reduction of $wg$.
If that sequence is leftward then this implies that the whole sequence
is to the left of the suffix $u'_2\cdots u'_k$. If it is rightward length
reducing
then Lemma~\ref{rrslem} tells us that its rightmost factor
must contain a critical subword, hence cannot intersect the suffix
$u'_2\cdots u'_k$ and must be to its left.
Hence (7) is proved.
\end{proof}

\begin{lemma}\label{lem2}
Suppose that $wg$ admits a leftward lex reducing sequence,
with corresponding factorisation $\alpha w_k\cdots w_1$ of $wg$, and
notation as above, and that $w$ admits no leftward lex reducing sequence.
Then for each $l$ with $1 \le l \le k$:
\begin{mylist}
\item[(1)] No proper prefix of $u_l$ is critical;\\
hence $u_l$ either has the form
${}_p(x,y)\xi (z^{-1},t^{-1})_n$ with $n>0$\\
or the form ${}_n(x^{-1},y^{-1})\xi (z,t)_p$ with $p>0$,\\ where
$\{x,y\}=\{z,t\} = \{a_{i_l},a_{j_l}\}$.
\item[(2)] $u_l'$ involves both of the generators
$a_{i_l}$ and $a_{j_l}$ when $l<k$.
\item[(3)] $p(u_l')+n(u_l') < m$ when $l<k$.
\item[(4)] When $l<k$, $u_l'$ begins with a letter whose name is not in
$\{ a_{i_{l+1}}, a_{j_{l+1}}\}$.
\item[(5)] When $1<l<k$, $u_l'$ ends with a letter whose name is not in
$\{ a_{i_{l-1}}, a_{j_{l-1}}\}$, and $u_1'$ ends with a letter
with a different name from $g$.
\item[(6)] When $k>1$, $w_1,\ldots w_{k-1}$  are maximal 2-generator subwords
of $wg$.  and $u'_1,\ldots u'_{k-1}$ are maximal 2-generator subwords of its
reduction $\alpha \tau(u_k) \cdots u'_1$.
\item[(7)] If $\alpha \tau(u_k)\cdots u'_1$ admits a further left lex reducing
or right length reducing
sequence, then all of the factors of that sequence, as well as the tail if it is
length reducing,
are contained within $\alpha \tau(u_k)$.
\end{mylist}
\end{lemma}

\begin{proof}
This is very similar to the previous proof, so we shall omit it.

Note, however, that in (1) the fact that $u_l$ has no critical
prefix follows from the lack of a left lex reducing sequence for $w$.

In the proof of (7) we consider of course the suffix
\[u'_{k-1} \cdots u'_2u'_1\]  of the reduction; otherwise the argument
is identical. 
\end{proof}

\begin{lemma}\label{extension_lemma}
Suppose that $w$ admits a rightward critical sequence
with corresponding factorisation $\alpha w_1\cdots w_k$, and
whose application to $w$ transforms it to a word ending in $g$.
Let $\zeta$ be a non 2-geodesic 2-generator word with $\f{\zeta}=g$,
for which $\s{\zeta}$ is 2-geodesic, and suppose that
$w\s{\zeta}$ is freely reduced.
Then the given sequence for $w$ extends to
a rightward length reducing sequence for $w\s{\zeta}$ of length $k+1$.
\end{lemma}
\begin{proof}
$\zeta$ is not 2-geodesic but some non-empty prefix of it is. Applying
Lemma~\ref{geodlem} to a maximal such prefix, we can deduce
that $\zeta$ contains a critical
subword $\theta$, such that replacement within $\zeta$ of $\theta$ by $\tau(\theta)$
gives a word with free cancellation between the last letter of $\tau(\theta)$
and the next letter of $\zeta$. Since $\s{\zeta}$ is geodesic, this 
substitution cannot happen with $\s{\zeta}$, and hence $\theta$ must
be a prefix of $\zeta$. So
$\theta=gw_{k+1}$, where $w_{k+1}$ is a prefix of $\s{\zeta}$. Now
$\alpha w_1\cdots w_kw_{k+1}\beta$ is a rightward critical factorisation
for $w\s{\zeta}$. The final application of $\tau$  (to $\theta$) in the
corresponding critical sequence
sparks a free reduction at the right hand end of $\theta$, and hence
this sequence is length reducing.
\end{proof}

We are now ready to prove our three propositions.

\begin{proofof}{Proposition~\ref{Wclosedngen}}

Since $w \in W$ and $wg \not\in W$, it follows from the definition of $W$ that
one of the following two possibilities occurs:

{\bf Case 1} $wg$ admits a rightward length reducing sequence enabling the
free cancellation of the final $g$.

{\bf Case 2} $wg$ admits a leftward lex reducing sequence but no rightward length
reducing sequence.

In each of the two cases we need to eliminate the possibilities that either
(a) the reduction of $wg$ admits a rightward length
reducing sequence, or
(b) the reduction of $wg$ admits a leftward
lex reducing sequence.
We use the notation for rightward and leftward reducing sequences that
was established above.

In Case 1,
we choose an optimal rightward length reducing sequence of $wg$,
with corresponding factorisation $\alpha w_1\cdots w_kg$;
recall that we call the word  resulting from this reduction $\rho_1(wg)$.
In Case 2,
we choose an optimal leftward lex reducing sequence of $wg$,
with corresponding factorisation $\alpha w_k \cdots w_1$;
recall we call the word resulting from this reduction $\rho_2(wg)$.
Note that we have defined $\rho(wg)$ to be $\rho_1(wg)$ in Case 1, and $\rho_2(wg)$
in Case 2.

We shall see that in Case (1), if $\rho_1(wg)$ admits either a rightward or
leftward reducing sequence, then the same is true of $w$,
while in Case (2), if $\rho_2(wg)$ admits a rightward reducing
sequence, then so does $wg$ (and so in fact we are in case (1)),
and if $\rho_2(wg)$ admits a leftward reducing sequence then either the
same is true of $w$ or $wg$ admits a leftward reducing sequence
whose left hand end is further left than in the previously chosen sequence for
$wg$, contradicting its optimality.
The details of thise argument now follow.

{\bf Case 1(a):}

Suppose that we are in Case 1 and that
$\rho(wg)=\rho_1(wg)$ admits a rightward length reducing sequence
with associated factorisation
$\beta \bar{w}_1\cdots \bar{w}_{\bar{k}}h\gamma$, where $h$ is the tail,
which cancels after application of the $\tau$-moves to $\rho(wg)$.

Since $w$ is in $W$ and hence cannot admit a rightward length reducing sequence,
the subword 
$\bar{w}_1\cdots \bar{w}_{\bar{k}}h$ of $\rho(wg)$ cannot be a subword of $w$.
Hence it has some intersection with the suffix $u'_1\cdots u'_k$ of $\rho(wg)$.
However, Lemma~\ref{lem1} (7) tells us that it is contained within
$\alpha u'_1$.
So the 2-generator subword $\bar{w}_{\bar{k}}h$ has some intersection with
$u'_1$, but by Lemma~\ref{lem1}(6) any other factors of this sequence are to the left of $u'_1$ in
$\rho(wg)$.
If $\bar{k}>1$, $\bar{w}_{\bar{k}}$ starts no later than
$\f{u_1'}$,
but if $\bar{k}=1$, $\bar{w}_1$ may start within $u_1'$.

We eliminate first the case $\bar{k}=1$.
We define $\eta$ to be the 2-generator subword of $\rho(wg)$ that starts at the 
beginning of $\bar{w}_1$ and ends at the right hand end of $u'_1$.
Then we define $\zeta$ be the 2-generator subword of $w$ that starts
at the beginning of $\bar{w}_1$ if that is within $\alpha$, or otherwise at
the beginning of $w_1$, and ends at the right hand end of $w_1$.
Since the application of a $\tau$-move to $\bar{w}_1$ sparks a free reduction
with the following letter in $u'_1$, $\eta$ cannot be 2-geodesic and, since
$\eta$ is a subword of a word obtained by applying a $\tau$-move to $\zeta$,
neither is $\zeta$. But $\zeta$ is a subword of $w$, so we contradict
$w \in W$.
 
So now we assume that $\bar{k}>1$. Then 
$\beta \bar{w}_1\cdots \bar{w}_{\bar{k}-1}$ is a 
rightward critical factorisation of length $\bar{k}-1$ of a word that
is also a prefix of $w$.
We shall now show how to extend this to yield a rightward length reducing
sequence of length $\bar{k}$
for $w$,
thereby contradicting $w \in W$.

Let $\bar{v}$ be the word that is derived from $\rho(wg)$ by applying the
$\bar{k}-1$ $\tau$-moves of this rightward critical sequence of length $\bar{k}-1$.
Then (using the notation we have already established for a
rightward critical factorisation of $\bar{w}$)
\[\bar{v} = \beta \bar{u}'_1 \bar{u}'_2 \cdots
\bar{u}'_{\bar{k}-1} \bar{u}_{\bar{k}} h \gamma.\]
Let $v$ be the word that is derived from $w$ by applying the same
sequence of $\bar{k}-1$ moves.
Then  $v$ and $\bar{v}$ share a prefix that includes
\[ \beta \bar{u}'_1 \bar{u}'_2 \cdots \tau(\bar{u}_{\bar{k}-1}).\]

Figure \ref{fig3} illustrates this situation. In the figure we can
trace out the paths of $wg, \rho(wg),\bar{v},v$. All four paths
pass through the circled vertex;
$wg$ and $\rho(wg)$ come into the circled vertex along the upper route 
along $\bar{w}_{\bar{k}-1}$ and part of $\bar{w}_{\bar{k}}$, while
$v$ and $\bar{v}$ follow the lower route
along $\tau(\bar{u}_{\bar{k}-1})$ and part of $\bar{w}_{\bar{k}}$.
The paths of $wg$ and $v$ leave the circled vertex along $w_1$,
while those of $\rho(wg)$ and $\bar{v}$ leave along $u'_1$.
\setlength{\unitlength}{0.74pt}
\begin{figure}
\begin{center}
{\large
\begin{picture}(750,130)(-150,-10)

\qbezier(-100,70)(-70,75)(-56,66) \put(-170,60){$wg, \rho(wg)$}
\qbezier(-100,30)(-10,120)(50,60) \put(-140,20){$v,\bar{v}$}
\put(-25,95){$\bar{w}_{\bar{k}-1}$}
\put(-24,68){$\bar{u}_{\bar{k}-1}$}
\qbezier(-79,49)(-30,20)(0,30) \qbezier(0,30)(25,40)(35,49)
\put(-50,40){$\tau(\bar{u}_{\bar{k}-1})$}
\put(-45,10){$\bar{u}'_{\bar{k}-1}$}

\put(70,70){$\bar{w}_{\bar{k}}$}
\put(65,74){\vector(-3,-2){14}}
\put(94,68){\vector(3,-2){38}}
\qbezier(35,50)(75,80)(100,50) 
\qbezier(35,49)(75,79)(100,49) 
\qbezier(36,49)(75,78)(100,48) 
\put(76,40){$\bar{u}_{\bar{k}}$}
 
\put(50,5){$\tau(\bar{u}_{\bar{k}})$}
\qbezier(35,50)(50,0)(138,34) 

\put(100,50){\circle*{10}}
\qbezier(100,50)(175,120)(249,80) 
\qbezier(100,49)(175,119)(248,80) 
\qbezier(100,48)(175,118)(247,80) 
\put(180,100){$w_1$} \put(180,80){$u_1$}

\qbezier(100,50)(130,25)(180,40) \qbezier(180,40)(200,50)(230,70)
\qbezier(100,49)(130,24)(180,39) \qbezier(180,39)(200,49)(230,69)
\qbezier(100,48)(130,23)(180,38) \qbezier(180,38)(200,48)(230.5,69)
\put(140,45){$\tau(u_1)$}
\put(137,15){$u'_1$}
\put(125,20){\vector(-1,1){20}}
\put(165,20){\vector(3,2){55}}
\qbezier(230,70)(295,110)(370,80)  \put(300,100){$w_2$}
\put(300,80){$u_2$}
\qbezier(230,70)(260,25)(300,40) 
\qbezier(300,40)(320,50)(350,60)
\put(260,45){$\tau(u_2)$}
\put(262,18){$u'_2$}
\end{picture}
}
\caption{\label{fig3} Collision between two rightward sequences.}
\end{center}
\end{figure}
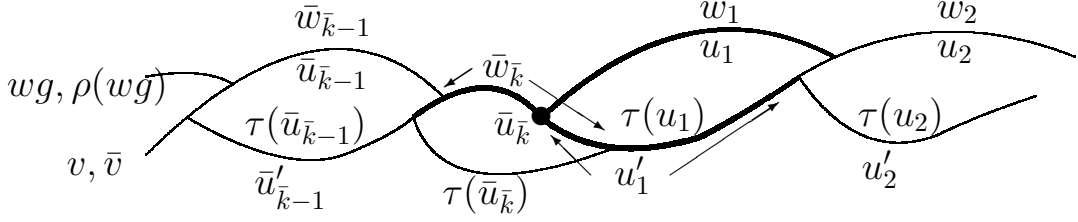

Now let $\eta$ be the 2-generator subword of $\bar{v}$ that starts at the 
beginning of $\bar{u}_{\bar{k}}$ and ends at the right hand end of $u'_1$.
Let $\zeta$ be the 2-generator subword of $v$ that starts at the beginning of
$\bar{u}_{\bar{k}}$ and ends at the right hand end of $u_1$. 
The subwords $\zeta$ and $\eta$ are marked in bold in the figure.

The first letter of both $\eta$ and $\zeta$ is the last letter of $\tau(\bar{u}_{\bar{k}-1})$.
Since the final move in the rightward length reducing sequence for $\rho(wg)$
sparks a free reduction, $\eta$ is not 2-geodesic, and since $\eta$ is
a subword of a word derived from $\zeta$ by applying a $\tau$-move to a suffix,
neither is $\zeta$.  The subword $\s{\zeta}$ of $w$ must be 2-geodesic, for 
otherwise
Theorem \ref{artin2} tells us that $\s{\zeta}$ is not in $W$, and hence neither
is $w$, and we have a contradiction. 
So now we can apply Lemma~\ref{extension_lemma} to deduce the existence of
a rightward length reducing sequence of length $\bar{k}$ for 
the prefix $\beta \bar{w}_1\bar{w}_2\cdots \bar{w}_{\bar{k}-1}\s{\zeta}$ 
of $w$,
contradicting the fact that $w \in W$.

{\bf Case 1(b):}

Next suppose that we are in Case 1 and that $\rho(wg)=\rho_1(wg)$ admits a
leftward lex reducing sequence with associated
factorisation $\beta \bar{w}_{\bar{k}}\cdots \bar{w}_1\gamma$.
Applying Lemma~\ref{lem1}(7) we see that
$\bar{w}_1$ is contained within $\alpha u'_1$ in $\rho(wg)$. Since
$w$ is in $W$ and so cannot admit a leftward lex reducing
sequence, $\bar{w}_1$ cannot be contained within $\alpha$, but must end within
$u'_1$.

Now we assume that $w_1 =u_1= {}_p(x,y) \xi (z^{-1},t^{-1})_n$ with
$p(w_1)+n(w_1)=m$ for the appropriate $m$, and $\tau(u_1) =
{}_n(x^{-1},y^{-1}) \delta(\xi) (t,z)_p$.
(We omit the argument that excludes the other choice for $w_1$ of
Lemma~\ref{lem1} (1), which is very similar.)
By Lemma~\ref{lem1} (1), we have $p>0$.

Now since the chosen factorisation of $wg$ is optimal, no proper
suffix of $u_1$ is critical, and so
$p( \s{u_1}) < p$ and hence
$p(\p{\tau(u_1)}) < p$; that is, $p(u'_1) < p$.
Hence if $\pi$ is the positive alternating subword of length $p$
at the beginning or end of $\bar{w}_1$, $\pi$ cannot be a subword
of $u'_1$ and so must intersect $\alpha$. 

If $n>0$, then $u'_1$ begins with a negative alternating subword,
and so $\pi$ is contained within $\alpha$. In this case we define
$\bar{w}'_1$ to be the subword of $w$ that starts at the beginning of $\pi$
and ends at the end of $w_1$.
If $n=0$, then by Lemma~\ref{lem2} (1) we can assume that $\pi$ lies
at the right hand end of $\bar{w}_1$, and so it must intersect $u'_1$,
and hence the prefix $\delta(\xi)$ of $u'_1$ (since it also
intersects $\alpha$).
In this case we define
$\bar{w}'_1$ to be the subword of $w$ that starts at the beginning of $\pi$
and ends at the end of the prefix ${}_p(x,y)$ of $w_1$.
Either way, $\bar{w}'_1$ is a critical subword of $w$,
and $\beta\bar{w}_{\bar{k}}\cdots \bar{w}_2\bar{w}'_1$
is a factorisation of a prefix of $w$ (either $\alpha w_1$ or a prefix of that)
corresponding to a leftward reducing sequence for that prefix.
This contradicts the fact that $w \in W$.

This completes the analysis of Case 1, so now suppose that we are in Case 2.

{\bf Case 2(a):}

The possibility that we are in Case 2, and that $\rho_2(wg)$ admits a
rightward length reducing sequence
is excluded by the following result, which we state as a separate lemma since we shall also use it in the proof of Proposition~\ref{ac1}:
\begin{lemma}\label{lem3}
Suppose that $w \in W$, and that  $wg$ admits an
optimal leftward lex reducing sequence with associated factorisation
$wg =\alpha w_k\cdots w_1$, leading to
\[\rho_2(wg) = \alpha \tau(u_k)u'_{k-1}\cdots u'_3 u'_2u'_1.\] 
Then $\rho_2(wg)$ admits a rightward length reducing sequence if and only if
$wg$ admits a rightward length reducing sequence.
\end{lemma}
We apply the lemma (whose proof we defer until the end of the proof of this
proposition) to deduce that in this case
$wg$ must also admit a rightward length reducing
sequence, a possibility that we have excluded from Case 2. 

{\bf Case 2(b):}

So now suppose that we are in Case 2 and that
$\rho(wg) = \rho_2(wg)$ admits a leftward lex reducing sequence
with associated factorisation $\beta \bar{w}_{\bar{k}}\cdots
\bar{w}_1\gamma$. Lemma~\ref{lem2} (7) tells us that
the subword $\bar{w}_1$ is a subword of $\alpha \tau(u_k)$ within
$\rho(wg)$.
Since $w \in W$, $\bar{w}_1$ cannot be a subword of $\alpha$
and so $\bar{w}_1$ must end within $\tau(u_k)$.

We suppose that $u_k = {}_p(x,y) \xi (z^{-1},t^{-1})_n$ with
$p(u_k)+n(u_k)=m$ for the appropriate $m$, and $\tau(u_k) =
{}_n(x^{-1},y^{-1}) \delta(\xi) (t,z)_p$ (We omit the other case dealing with the
other possibility for $u_k$  of Lemma~\ref{lem2} (1), which is similar.)
By Lemma~\ref{lem2} (1), we have $n > 0$.
Let $\nu$ be the negative alternating sequence of length $n$ at the beginning or
end of $\bar{w}_1$.  and let $\nu'$ be the subword ${}_n(x^{-1},y^{-1})$ of
$\tau(u_k)$.

We claim that $\nu'$ must be the unique negative alternating
subword of length $n$ in $\tau(u_k)$.
If $p=0$, then this is true by definition of critical words for negative words.
If $p>0$ and there there was another such subword, then it
would necessarily lie entirely within $\delta(\xi)$, in which case $\xi$ would
also contain such a subword, and then a prefix of the subword ${}_p(x,y) \xi$
of $w$ would be upper critical. The application of $\tau$ to this prefix
would give $w$ a leftward lex reducing sequence of length 1,
contradicting $w \in W$. Hence in this case too the claim is proved.

Suppose first that $p > 0$.

If $\nu \ne \nu'$ then, by the preceding paragraph,
$\nu$ lies to the left of $\nu'$ and hence to the left of
$\tau(u_k)$, at the beginning of $\bar{w}_1$, within $\alpha$.
Now
we define $\bar{w}'_1$  to be the subword of $\alpha u_k$ that runs from the
beginning of $\nu$ to the end of the prefix ${}_p(x,y)$ of $u_k$,
and find a leftward lex reducing sequence for $w$
with associated factorisation
$\beta \bar{w}_{\bar{k}}\cdots \bar{w}_2\bar{w}_1'\gamma'$,
contradicting $w \in W$.

So we suppose that $\nu=\nu'$. 
If $\nu$ is at the beginning of $\bar{w}_1$,
then $\tau(\bar{w}_1)$ has the same prefix ${}_p(x,y)$ 
as $u_k$ and then $\tau(u_k) \lexle u_k$ implies
$\bar{w}_1 \lexle \tau(\bar{w}_1)$, so we must have $\bar{k}>1$. 
But then then also $\f{\tau(\bar{w}_1}= \f{u_k} = \f{w_k}$ and so
$\beta \bar{w}_{\bar{k}}\cdots \bar{w}_2\f{\tau(\bar{w}_1)}
=\beta \bar{w}_{\bar{k}}\cdots \bar{w}_2\f{w_k}$
is a prefix of $\alpha w_k$ and hence of $w$.
Then, where $\bar{w'}_2 = \bar{w}_2\f{w_k}$, the factorisation
$\beta \bar{w}_{\bar{k}}\cdots \bar{w'}_2$ of that prefix
is associated with a leftward lex reducing sequence that also reduces $w$,
contradicting $w \in W$.

On the other hand if $\nu$ is at the right hand end of $\bar{w}_1$, 
then $\bar{w}_1 = w_1'\nu$.
Then there is a leftward lex reducing sequence of $wg$ with
factorisation $\beta \bar{w}_{\bar{k}}\cdots
\bar{w}_2 \bar{w}_1'w_{k-1}\cdots w_1$ in which $\bar{w}'_1 = w'_1w_k$.
This extends further left that the chosen factorisation, contrary to
assumption.

If $p=0$ then by Lemma~\ref{lem2} (1)
applied to the shortest prefix of $\rho(wg)$ that is not in $W$, $\nu$ must be
at the right hand end of $\bar{w}_1$,
and again $wg$ has a leftward reducing sequence that extends further left than
the chosen one, giving a contradiction as before.

\end{proofof}

To complete the proof of Proposition~\ref{Wclosedngen}.
we need the proof of Lemma~\ref{lem3}.

\begin{proofof}{Lemma~\ref{lem3}}
We prove first  
(a) that if $\rho_2(wg)$ admits a rightward
length reducing sequence then $wg$ admits one too,
and then (b) that if $wg$ admits a rightward length reducing sequence,
then so does $\rho_2(wg)$.

{\bf Proof of (a):}

Suppose that $\rho_2(wg)$ has a rightward length reducing sequence with
associated factorisation $\beta \bar{w}_1\cdots \bar{w}_{\bar{k}}h\gamma$,
where the generator $h$ cancels after application of the $\tau$-moves to
$\bar{w}:=\bar{w}_1\cdots \bar{w}_{\bar{k}}$. Then by Lemma~\ref{lem2} (7)
$\bar{w}_{\bar{k}}h$  is a subword of $\alpha \tau(u_k)$.
If it were also a subword of $\alpha$, we would have a rightward
length reducing sequence for $w$, contradicting the fact that $w \in W$.
Hence $\bar{w}_{\bar{k}}h$ must end within $\tau(u_k)$.
But by Lemma~\ref{lem2}(6) any other factors of this sequence must be within
$\alpha$.

The proof is now by induction on $k$. 

{\bf Base case.}
Suppose that $k=1$.

In the case where $\bar{k}=1$, we 
define $\eta$ to be the maximal 2-generator subword of $\rho(wg)$ that
contains $\bar{w}_1h$.
Since the application of a $\tau$-move to $\bar{w}_1$ 
enables a free reduction, $\eta$ cannot be 2-geodesic.
Hence neither is $\zeta$,
the 2-generator subword of $w$ which is mapped to $\eta$ by applying a
$\tau$-move to a subword.
So $\zeta$ admits a right length reducing sequence of length 1,
and hence so does $wg$.

So now we shall assume that $\bar{k}$ is not 1.
Let $\bar{v}$ be the word obtained from $\rho_2(wg)$ by applying the
first $\bar{k}-1$ terms of its rightward length reducing sequence,
and let $v$ be the word obtained by applying the same sequence of
 moves to $wg$.

Figure \ref{fig4} illustrates this situation. 
The circled vertex
marks the end of the common prefix of $v,\bar{v}$.
The subwords $\zeta$ and $\eta$ (defined below) are marked in bold in the figure.
\setlength{\unitlength}{0.75pt}
\begin{figure}
\begin{center}
{\large
\begin{picture}(760,150)(-150,-30)

\qbezier(-100,70)(-70,75)(-56,66) \put(-170,60){$wg, \rho_2(wg)$}
\qbezier(-100,30)(-10,120)(50,60) \put(-140,20){$v,\bar{v}$}
\put(-25,95){$\bar{w}_{\bar{k}-1}$}
\put(-23,68){$\bar{u}_{\bar{k}-1}$}
\qbezier(-79,49)(-30,20)(0,30) \qbezier(0,30)(25,40)(35,50)
\put(-40,40){$\tau(\bar{u}_{\bar{k}-1})$}
\put(-45,10){$\bar{u}'_{\bar{k}-1}$}

\put(70,70){$\bar{w}_{\bar{k}}$}
\put(65,74){\vector(-3,-2){14}}
\put(94,68){\vector(3,-2){40}}
\qbezier(35,50)(75,80)(100,50) 
\qbezier(35,49)(75,79)(100,49) 
\qbezier(36,49)(75,78)(100,48) 
\put(76,40){$\bar{u}_{\bar{k}}$}
 
\put(50,6){$\tau(\bar{u}_{\bar{k}})$}
\qbezier(35,50)(50,0)(136,34) 

\put(100,50){\circle*{10}}
\qbezier(100,50)(175,120)(250,80) 
\qbezier(100,49)(175,119)(250,79) 
\qbezier(100,48)(175,118)(250,78) 
\put(160,100){$w_1$} \put(160,75){$u_1$}

\qbezier(100,50)(130,25)(180,40) \qbezier(180,40)(200,50)(250,80)
\qbezier(100,49)(130,24)(180,39) \qbezier(180,39)(200,49)(250,79)
\qbezier(100,48)(130,23)(180,38) \qbezier(180,38)(200,48)(250,78)
\put(140,45){$\tau(u_1)$}
\put(140,15){$h$}
\put(65,120){\vector(-1,-2){30}}
\put(100,140){\vector(3,-1){130}}
\put(75,130){$\zeta$}
\put(75,-35){$\eta$}
\put(65,-30){\vector(-1,2){30}}
\put(100,-30){\vector(3,2){130}}
\end{picture}
}
\caption{\label{fig4} Part (a). Induction, case $k=1$.}
\end{center}
\end{figure}
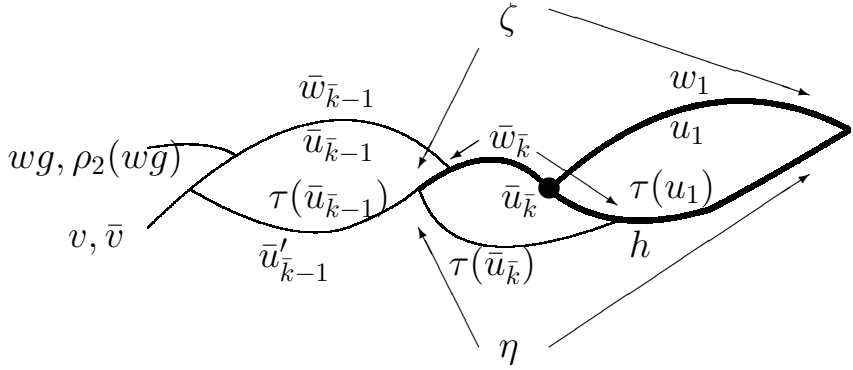

Let $\eta$ be the 2-generator suffix of $\bar{v}$ that starts 
at the left hand end of $\bar{u}_{\bar{k}}h$.
The final $\tau$-move of the rightward sequence, which is applied to the prefix
$\bar{u}_{\bar{k}}$
of $\eta$, enables a free reduction, so $\eta$ is not 2-geodesic.
So the word $\zeta$ obtained by replacing 
the subword $\tau(u_1)$ in $\eta$ by $u_1$ is also not 2-geodesic.
Now we can apply Lemma~\ref{extension_lemma} to get a rightward
length reducing sequence for 
$wg= \beta \bar{w}_1\bar{w}_2\cdots \bar{w}_{\bar{k}-1}\s{\zeta}$.

{\bf Inductive step.} Suppose that $k>1$.
Fig~\ref{fig5} illustrates this part of the proof.
\setlength{\unitlength}{0.48pt}
\begin{figure}
\begin{center}
{
\small
\begin{picture}(800,300)(-840,-10)
\qbezier(-100,50)(-175,120)(-250,80) 
\put(-180,100){$w_1$}
\put(-180,76){$u_1$}
\qbezier(-100,50)(-130,25)(-180,40) \qbezier(-180,40)(-200,50)(-250,80)
\put(-235,75){$g'$}
\put(-185,46){$\tau(u_1)$}
\put(-180,15){$u'_1$}
\qbezier(-250,80)(-295,110)(-370,80) \put(-300,100){$w_2$}
\qbezier(-250,79)(-295,109)(-370,79) \qbezier(-250,78)(-295,108)(-370,78)
\put(-300,80){$u_2$}
\qbezier(-231,68)(-260,25)(-300,40) \qbezier(-300,40)(-320,50)(-370,80)
\put(-305,48){$\tau(u_2)$}
\put(-300,15){$u'_2$}
\qbezier(-370,80)(-415,110)(-490,80) \put(-420,100){$w_3$}
\qbezier(-370,79)(-415,109)(-490,79) \qbezier(-370,78)(-415,108)(-490,78)
\put(-420,80){$u_3$}
\qbezier(-351,68)(-380,25)(-420,40) \qbezier(-420,40)(-440,50)(-490,80) 
\put(-425,48){$\tau(u_3)$}
\put(-420,15){$u'_3$}
\qbezier(-490,80)(-535,110)(-610,80) \put(-540,100){$w_4$}
\qbezier(-490,79)(-535,109)(-610,79) \qbezier(-490,78)(-535,108)(-610,78)
\put(-540,80){$u_4$}
\qbezier(-471,68)(-500,25)(-540,40) \qbezier(-540,40)(-560,50)(-610,80) 
\put(-545,48){$\tau(u_4)$}
\qbezier(-610,80)(-655,110)(-730,80) \put(-660,100){$w_k$}
\qbezier(-610,79)(-655,109)(-730,79) \qbezier(-610,78)(-655,108)(-730,78)
\put(-660,80){$u_k$}
\qbezier(-591,68)(-620,25)(-660,40) \qbezier(-660,40)(-680,50)(-728,80) 
\put(-665,48){$\tau(u_k)$}
\put(-780,70){$\alpha$}
\qbezier(-730,80)(-780,45)(-830,80) \qbezier(-730,79)(-780,44)(-830,79)
\qbezier(-730,78)(-780,43)(-830,78)
\put(-500,200){$w'$}
\qbezier(-250,80)(-500,300)(-830,80)
\end{picture}
}
\caption{\label{fig5} Part (a), Inductive step.}
\end{center}
\end{figure}
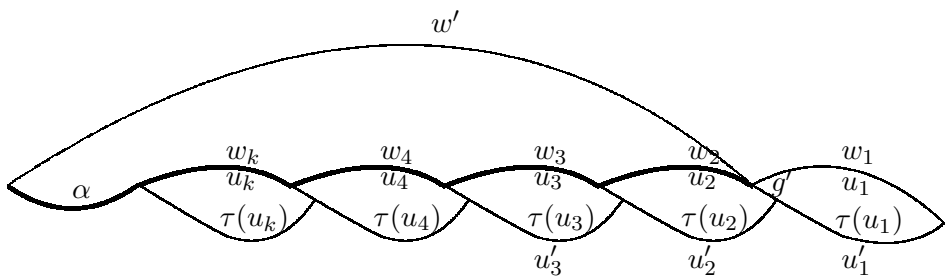

Let $w'$ be the prefix
$\alpha w_k \cdots w_2$ of $w$; as a prefix of $w$ it must be in $W$.
Let $g' :=\f{\tau(w_1)}=\f{\tau(u_1)}$.
The word  $ \alpha w_k \cdots w_2\tau(w_1)$ is the result of
the first of the $k$ steps of the leftward reduction of $wg$, and so
admits a leftward lex reducing sequence of length $k-1$;
the same leftward lex reducing sequence of length $k-1$
reduces $w'g'$ (as a prefix of the above)
to a prefix $\rho_2(w'g')$ of $\rho_2(wg)$.

Now the rightward length reducing sequence
that we have for $\rho_2(wg)$ stops within the $\tau(u_k)$ subword, and so
certainly to the left of the final suffix $\s{ \tau(w_1)}$ of $\rho_2(wg)$;
hence $\rho_2(w'g')$ admits a rightward length reducing sequence. 

Now we can apply the induction hypothesis to $w'g'$ to deduce that
$w'g'$ admits a rightward length reducing sequence.
Since $w' \in W$, the last factor of the associated
factorisation is a suffix of $w'g'$.  The sequence transforms $w'g'$
to a word $w''g'$, where
$\l{w''}$ is the inverse of $g'$.
The same rightward critical sequence can be applied to $w=w'w_1$, which
it transforms to $w''w_1$. 
Finally, we consider the suffix $\zeta = {g'}^{-1}w_1$ of $w''w_1$.
Since ${g'}^{-1}\tau(w_1)$ is not freely reduced, it is not 2-geodesic,
and hence neither is $\zeta$. Now, just as in the $k=1$ case we can
apply Lemma~\ref{extension_lemma} to find a rightward length reducing
sequence for $wg$.

{\bf Proof of (b):}

Now suppose that $wg$ admits a rightward length reducing sequence.
Again we use induction on $k$.

{\bf Base case.}
When $k=1$ the proof is very similar to the $k=1$ case above.
We just interchange the roles of $u_1=w_1$ and $\tau(u_1)$. But we observe that
in this case the tail of the factorisation of $wg$ must be the final $g$,
since $w \in W$. 

{\bf Inductive step.}
Now suppose that $k>1$. In this case by Lemma~\ref{lem2} $w_1$,\ldots
$w_{k-1}$ are maximal 2-generator words and geodesic.

Suppose that $wg$ admits a rightward length reducing sequence of length
$\bar{k}$. This cannot apply to $w$, since $w \in W$.
It cannot have length 1. For if it did, it would apply to the
suffix $w_1$, which is geodesic.
So $\bar{k}>1$ and the $(\bar{k}-1)$-th $\tau$-move must change $\l{w_2}$ to a letter
$h$, say, where $hw_1$ is not 2-geodesic. 
But then, by Lemma~\ref{geodlem},
$w_1$ must have a critical prefix $v_1$ such that $\tau(v_1)$
begins with $h^{-1}$; the possibility that
$\f{w_1}=h^{-1}$ is excluded by the fact that $w \in W$.  
But in fact for any critical prefix $v_1$ of $w_1$,
$\f{\tau(v_1)}=\f{\tau(w_1)}$, and so we have $g' := \f{\tau(w_1)}=h^{-1}$.
So the first $\bar{k}-1$ moves of the rightward length reducing sequence of $wg$
also
induce a rightward length reducing sequence of $w':=\alpha w_k \cdots w_2 g'$.
But $w'$ admits a leftward lex reducing sequence of length
$k-1$, and so we can now apply
our inductive hypothesis to conclude that 
$\rho_2(w')$
admits a rightward length reducing sequence.
The result immediately follows since
\[\rho_2(w') =\rho_2(\alpha w_k \cdots w_2 g') =
\alpha \tau(u_k)u'_{k-1}\cdots u'_3 u'_2 \]
is a prefix of $\rho_2(wg)$. 
\end{proofof}

\begin{proofof}{Proposition~\ref{ac1}}
This is immediate except in the case when $wg$ is freely reduced but
$wg \not\in W$, in which case $\rho(wg)$ is defined as in the proof of
Proposition~\ref{Wclosedngen}, and we use the same notation as in that
proof.

First we suppose that $\rho(wg)=\rho_1(wg).$
In this case $wg$ admits a factorisation $\alpha w_1\ldots w_kg$,
corresponding to a rightward length
reducing sequence. The sequence of $\tau$ moves
transforms $w$ to  $w' := \alpha u'_1u'_2\cdots u'_{k-1} \tau(u_k)$ 
using our standard notation associated with a rightward factorisation of $wg$,
with $\tau(u_k)$ ending in $g^{-1}$. Then the
final $g^{-1}$ is cancelled to produce
$\rho(wg) = \alpha u'_1u'_2u'_3 \cdots u'_k$.
So $\rho(wg)g^{-1} = w'$. 
Hence to complete consideration of this case, we need to show that $\rho(w')=w$.

It follows from Lemma~\ref{lem1}~(4) and (5), that
reversing the $\tau$-moves in the rightward length reducing sequence
for $wg$ results in a leftward lex reducing sequence
$\S$ that transforms $w'$ back to $w$. 
Our next step is to show that $\S$ is optimal.

So let $\S'$ be the optimal lefward lex reducing
sequence for $w'$, that is the leftward lex reducing sequence
for $w'$ that extends furthest to the left in $w'$. Then $\S'$
involves at least $k$ $\tau$-moves, and the first $k-1$ of those must
match the first $k-1$ $\tau$-moves of $\S$, since 
those must correspond to $\tau(u_k),u'_{k-1},\ldots,u'_2$, defined as maximal
2-generator subwords of $w'$ (as in Lemma~\ref{lem2} (6)).
These first $k-1$ moves 
transform $w'$ back to $\alpha \tau(u_1) w_2\cdots w_k$.
Suppose that $u_1 = {}_p(x,y) \xi (z^{-1},t^{-1})_n$ with
$p+n=m$ for the appropriate $m$, and $\tau(u_1) =
{}_n(x^{-1},y^{-1}) \delta(\xi) (t,z)_p$ (the other case is similar)
where, by Lemma~\ref{lem1} (1), $p \ne 0$.
If the next $\tau$-move in $\mathcal{S'}$ transforms $\tau(u_1)$ back to
$u_1$, then we are back to $w$, and any further $\tau$-moves in
$\mathcal{S'}$ could have been applied to $w$, contradicting $w \in W$.
Now if $\S'$ extends further left than $\S$,
the next $\tau$-move in $\mathcal{S'}$ must apply to a word
$\beta \tau(u_1)$ having $\tau(u_1)$ as a proper suffix.  Since $\beta\tau(u_1)$
is critical, $\beta$ (like $\tau(u_1)$) must have a negative alternating
word of length $n$ as a prefix. But in that case $\beta\,{}_p(x,y)$ must also
be critical, and is a subword of $w$. So this $\tau$-move followed by any
remaining moves in the sequence $\S'$ is a leftward lex reducing sequence for
$w$, contradicting $w \in W$. 
Hence $\S$ is indeed the optimal leftward lex reducing sequence that reduces $w'$
to $w$, that is $\rho_2(w')=w$.

Now we can apply Lemma~\ref{lem3} to see that if $w'$ can also be reduced
using a rightward length reducing sequence, then $w=\rho_2(w')$ must also
admit such a sequence. But this would contradict $w \in W$. Hence $w'$
admits no such reduction, and so we must have $\rho(w')=\rho_2(w')=w$
as required.

Now we suppose that $\rho(wg)=\rho_2(wg)$.
In that case we have a factorisation
$wg= \alpha w_k\cdots w_1$ of $wg$ corresponding to a leftward lex reducing
sequence for of $wg$ to 
\[\rho(wg)=\alpha \tau(u_k)u'_{k-1} \cdots u'_2u'_1.\]
Reversing these $\tau$-moves results in a rightward length reducing sequence
$\S$ for $\rho(wg)g^{-1}$, and we need to verify that there
is no alternative  rightward length reducing sequence
$\S'$ for $\rho(wg)g^{-1}$ that starts further to the right than
$\S$. By  Lemma~\ref{lem2} (7), such a sequence would have to
start to the left of $u'_{k-1}$, and so the factorisation would have
the form
\[\alpha \beta u''_ku'_{k-1} \cdots u'_2u'_1\]
with $\beta u_k'' = \tau(u_k)$ and $\beta$ nonempty.
Let $u_k = {}_p(x,y) \xi (z^{-1},t^{-1})_n$ with
$p+n=m$ for the appropriate $m$, and $\tau(u_k) =
{}_n(x^{-1},y^{-1}) \delta(\xi) (t,z)_p$ (the other case being similar)
where, by Lemma~\ref{lem2} (1), $n \ne 0$.
If $p>0$, then $n(\delta(\xi)) = n(\xi) = n$, in which case
the subword ${}_p(x,y) \xi$ of $w$ contains an upper critical subword,
contradicting $w \in W$. 
The case $p=0$ is ruled out by the definition of critical words in this case,
which requires that $\tau(u_k)$ contains a unique negative alternating
subword of length $n$.  
\end{proofof}

\begin{proofof}{Proposition~\ref{ac2}}
To ease the notation, let $a=a_i$, $b=a_j$, where we may assume that
$a \lexle b$, and $m = m_{ij}$.
We consider the 2-generator Artin group
$\DA_m=\langle a,b \mid {}_m(a,b)= {}_m(b,a) \rangle$.
Our general strategy is to show that in every situtation, in the
course of the computation of $\rho(w\,{}_m(a,b))$ by appending each
letter of ${}_m(a,b)$ in turn to $w$, at most one such appended letter
will precipitate a leftward lex reduction or a rightward length reduction of
the resulting word.  All other appended letters result either in no reduction,
or in the cancellation of the appended letter by free reduction.
In general, a similar
leftward or rightward reduction (if any) is involved in the computation of
$\rho(w\,{}_m(b,a))$, and we then apply Theorem~\ref{artin2} to $\DA_m$ to
infer the result. 

The result is clear if $w$ is empty or if
$w$ is a power of a letter whose name is not $a$ or $b$, for in these cases we
have $\rho(w\,{}_m(a,b)) = \rho(w\,{}_m(b,a)) = w\,{}_m(a,b)$.

Suppose that the name of $\l{ w }$ is $c$, with $c \not\in \{a,b\}$.

If $w$ does not have the form $w' v$ with $v$ a 2-generator word involving
$a,c$ or $b,c$, then again
$\rho(w\,{}_m(a,b)) = \rho(w\,{}_m(b,a)) = w\,{}_m(a,b)$. So from
now on we assume that
$w = w'v$ where $v$ involves $a$ and $c$ (the other case is similar).
In this case, if $\rho(wa) = wa$ then
$\rho(w\,{}_m(a,b)) = \rho(w\,{}_m(b,a)) = w\,{}_m(a,b)$.
So we suppose that $\rho(wa) \neq wa$.

Now we have the usual two cases for $\rho(wa)$.
In either case, by Lemmas~\ref{lem1}~(5) and~\ref{lem2}~(5),
the name of the final letter of
$\rho(wa)$ is $c$, so \[\rho(w\,{}_m(a,b)) = \rho(wa)\, {}_{m-1}(b,a).\]
Note also that $\rho(w\,{}_{m-1}(b,a)) = w\,{}_{m-1}(b,a)$. 

If we are in Case 1 for $\rho(wa)$, and $wa$ has a rightward length reducing
sequence with factorisation  $wa=\alpha w_1\cdots w_ka$, then
$w\,{}_m(b,a)$ has a rightward length reducing sequence with factorisation
$\alpha w_1\cdots w_kw_{k+1}x$, with $w_{k+1} = {}_{m-1}(b,a)$
and $x$ the final letter of $_{m}(b,a)$,
resulting in $u_{k+1} = a^{-1}\,{}_{m-1}(b,a)$,
$\tau(u_{k+1}) = {}_{m-1}(b,a) x^{-1}$, so
\[\rho(w\,{}_m(b,a)) = u'_1\cdots u'_k\, {}_{m-1}(b,a)= \rho(wa)\, {}_{m-1}(b,a),\]
and hence $\rho(w\,{}_m(b,a))= \rho(w\,{}_m(a,b))$, as required.  

Similarly, in Case 2 for $\rho(wa)$, where $wa$ has a leftward lex reducing
sequence with factorisation  $wa=\alpha  w_k\cdots w_1$,
$w\,{}_m(b,a)$ has a leftward lex reducing sequence with factorisation
$\alpha w_k\cdots \p{w_1}\, {}_m(b,a)$, resulting in
\[\rho(w\,{}_m(b,a)) = \alpha u'_k\cdots u'_1\, {}_{m-1}(b,a) = \rho(wa)\,
{}_{m-1}(b,a) = \rho(w\,{}_m(a,b)).\]

Now we suppose that the name of $\l{w}$ is $a$ or $b$.
Without loss of generality, we can assume that it is $a$;
although the other case appears to
be inequivalent, since $a \lexle b$, essentially the same arguments
work in both cases.  So $\l{w} = a$ or $a^{-1}$.
We have $w = w' v$, where $v$ is a word involving only $a$ and (possibly) $b$,
and $w'$ is either empty or else the name of $\l{w'}$ is not $a$ or $b$.
Let $p=p(v), n=n(v)$; so $p+n \le m$.
When $\l{w} = a$ or $a^{-1}$, respectively, let $v = v' (b,a)_k$  or
$v = v' (b^{-1},a^{-1})_k$ with $k$ maximal.

{\bf Case 1.}
Suppose first that $n<m$ and that $w'$ admits a rightward critical sequence that
transforms $w'$ to $w''$ where $\l{w''} \in \{a^{-1}, b^{-1}\}$
and $n(\l{w''}v) = n+1$. Then we must have $p+n < m$,
or else $w$ would admit a rightward length reducing sequence.

If $\l{w} = a$, then we find that $w\,\,{}_{m-n-1}(a,b)$ is critically reduced,
but $w\,{}_{m-n}(a,b)$ admits a rightward reducing sequence starting with the
sequence for $w'$. The remaining $n$ letters of $_m{}(a,b)$ then cancel with
a suffix of the reduction of $w\,\,{}_{m-n}(a,b)$,
and we get $\rho(w \,\,{}_m(a,b)) = \p{w''} v_1$ for some
critically reduced 2-generator word $v_1$ that is equal in $G(a,b)$ 
to $\l{w''}v\,\,{}_m(a,b)$. 
There is also a rightward length reducing sequence starting with the same
sequence for $w'$ for $w \,\,{}_{m-n-k}(b,a)$, following which the next
$n$ letters of ${}_m(b,a)$ cancel and, since $k \le p$ and $p+n<m$,
the final $k$ letters provoke no further reductions.
So we have $\rho(w \,\,{}_m(b,a)) = \p{w''} v_2$
with $v_2$ equal in $G(a,b)$ to $\l{w''}v\,\, {}_m(b,a)$. Since $v_1$ and $v_2$
are reduced 2-generator words representing the same element of
$G(a,b)$, Theorem~\ref{artin2} implies that they are equal, so
$\rho(w \,{}_m(a,b)) = \rho(w \,{}_m(b,a))$.

If $\l{w} = a^{-1}$, then $w\, {}_{m-n-1}(b,a)$ is critically reduced,
$w \,{}_{m-n}(b,a)$ admits a rightward length reducing
sequence starting with the sequence for $w'$,
and the remaining $n$ letters of ${}_m(b,a)$ cancel.
There is also a rightward length reducing sequence starting with the same
sequence for $w'$ for $w \,{}_{m-n+k}(a,b)$, following which the
remaining $n-k$ letters of ${}_m(a,b)$ cancel, and the result
follows as in the previous case.

{\bf Case 2.} Suppose then $n=m$ or that $w'$ admits no such rightward critical
sequence. If $\l{w} = a$, then again
$w \,{}_{m-n-1}(a,b)$ is critically reduced, and $w \,{}_{m-n}(a,b)$ admits no
rightward length reducing sequence, but it may admit a leftward lex reducing sequence.
If so, then the remaining $n$ letters of ${}_m(a,b)$ cancel.
In that case, $w \,{}_{m-n-k}(b,a)$ admits a corresponding leftward
lex reducing sequence, following which the next $n$ letters of ${}_m(b,a)$
cancel. Now, since $w \in W$, we must have $k<m-n$ in this
situation, so the final $k$ letters of ${}_m(b,a)$ provoke no further
reductions. So, as in Case 1, we can apply Theorem~\ref{artin2} to
conclude that $\rho(w \,{}_m(a,b)) = \rho(w \,{}_m(b,a))$.

Suppose, on the other hand, that $w \,{}_{m-n}(a,b)$ is critically reduced.
If $n=0$, then $w\,{}_{m-1}(b,a)$ must be critically reduced (because,
if not, then a corresponding reduction could be applied to
$w \,{}_{m-n}(a,b)$), and we have $\rho(w\,{}_m(b,a)) = w\,{}_m(a,b)$.
If $n>0$, then $w \,{}_{m-n+1}(a,b)$ admits a rightward length reducing
sequence of length 1, and the remaining $n-1$ letters of ${}_m(a,b)$ cancel.
Similarly,  $w \,{}_{m-n-k+1}(b,a)$ admits a corresponding rightward length
reducing sequence, and the following $n-1$ letters of ${}_m(b,a)$ cancel.
The final $k$ letters of ${}_m(b,a)$ can provoke no further reductions,
since such a reduction could only result from the final letter in the case
$k=m-n$, but if there were such a reduction then the original word $w$
would admit a corresponding reduction, contradicting $w \in W$.
So the result follows as before in this case.

If $\l{w} = a^{-1}$, then $w \,{}_{m-n-1}(b,a)$ is critically reduced,
and $w \,{}_{m-n}(b,a)$ admits no rightward length reducing sequence.
If $w \,{}_{m-n}(b,a)$ admits a leftward lex reducing sequence, then the
remaining $n$ letters of ${}_m(b,a)$ cancel.
In that case $w \,{}_{m-n+k}(a,b)$ admits a corresponding leftward
lex reducing sequence, and the remaining $n-k$ letters of ${}_m(b,a)$ cancel,
and the result follows as before.

If, on the other hand, $w \,{}_{m-n}(b,a)$ is critically reduced
(note that this occurs, in particular, when $m=n$), then
$w \,{}_{m-n+1}(b,a)$ admits a rightward length reducing sequence of length 1,
as does $w \,{}_{m-n+k+1}(a,b)$, and again the result follows.
\end{proofof}

\section{Geodesics in Artin groups of large type}
\label{geodesics_sec}
\begin{theorem}\label{fftp:thm} Artin groups of large type on their standard generating sets
satisfy FFTP, and hence the set of geodesic words is regular.
\end{theorem}

The rest of this section is devoted to the proof of this theorem.
Throughout this section, $G$ will be an Artin group of large type over $X$,
and $W$ the set of shortlex minimal representatives of its elements.
We start with a useful technical result.

\begin{lemma}\label{genandinv}
If $w \in W$, $x \in X$ and $wx$ and $wx^{-1}$ are both freely reduced,
then $wx$ and $wx^{-1}$ cannot both
be non-geodesic.
\end{lemma}
\begin{proof} 
We use induction on $|w|$. The result is clear if
$w$ involves at most two generators because it is easily seen that
$p(wx)+n(wx)>m$ and $p(wx^{-1})+n(wx^{-1})>m$ cannot both hold,
given that $w$ is geodesic.
Otherwise, if $wx$ and $wx^{-1}$ are both non-geodesic, then
Proposition~\ref{Wclosedngen} implies that $wx$ and $wx^{-1}$ both admit
rightward length reducing sequences. It follows from the 2-generator case that
these sequences cannot both have length 1.

Suppose that one of these sequences, the one for $wx$ say, has length 1,
and the other has length greater than 1.
Let $w_1$ be the result of applying all $\tau$-moves except for the last
in the reduction sequence for $wx^{-1}$, and let $u_1$ be the maximal
2-generator suffix of $w_1$. Then $u_1x$ and $u_1x^{-1}$ are both non-geodesic,
so the result again follows from the 2-generator case.

Finally, suppose that both sequences have length greater than 1, and let
$w = \alpha u$, where $u$ is the maximal 2-generator suffix of $w$.
Then applying all terms except the last in the reduction sequences for
$wx$ and $wx^{-1}$ transforms $\alpha$ to words with last letters $g$ and
$h$, where $gux$ and $hux^{-1}$ are 2-generator words with
$p(gux)+n(gux)>m$ and $p(hux^{-1})+n(hux^{-1})>m$, but all proper subwords
of $gux$ and $hux^{-1}$ are geodesic.
Suppose without loss of generality that $\l{u} \in X$. Then since
$p(hux^{-1}) = p(hu)$, we must have $n(hux^{-1}) > n(hu)$, which is only
possible if $n(hux^{-1})=1$, $p(u)=m-1$ and $p(hu) = m$. So we must
have $h \in X$ and $h \ne \f{u}$. Similarly, we find that $p(ux)=m$ and
$n(gux)=1$, so $g \in X^{-1}$. But we cannot have $g = \f{u}^{-1}$, and
so we must have $g=h^{-1}$.
But then $\alpha g$ and $\alpha g^{-1}$ are both non-geodesic, and freely
reduced, by our definition of $\alpha$, and the
result follows by the inductive hypothesis applied to $\alpha$.
\end{proof}

In order to prove the theorem we need to examine in detail the process of
reduction of a geodesic word $v$ to its shortlex minimal representative
$\rho(v)$, and prove a number of technical results. 
We shall use all the notation we established in the previous sections, and 
introduce some more.

The reduction is done in at most $n:=|v|$ steps, through a sequence of
words $v^{(0)}=v,v^{(1)},\cdots ,v^{(n)}=\rho(v)$; for each $i$ from
$1$ to $n$, $v^{(i)}$ is either equal to $v^{(i-1)}$ or is derived from it by replacing its prefix of
length $i$ by its lex reduction.
When $v^{(i)} \neq v^{(i-1)}$, Proposition~\ref{Wclosedngen} says that
the reduction is through a single leftward lex
reducing sequence of which the first $\tau$-move is applied to a word ending
at the $i$-th letter of $v^{(i-1)}$.

In general we assume that $v$ involves at least three generators (the 2-generator
case being dealt with in Section~\ref{dihedral_sec}).
In that case, we define $u$ to be the maximal 2-generator suffix of $v$, and let $a,b$ be the
names of the two generators involved in $u$.
Similarly for each $i$ we define $u^{(i)}$ to be the maximal
suffix of $v^{(i)}$ involving $a$ and $b$ (conceivably $u^{(i)}$ might be empty or
involve just one of those two generators).
Then $v = \alpha gu$ with $g \in A$, where the name of $g$ is neither $a$ nor $b$.
Let $k := |\alpha g|$; so $v^{(k)}=\rho(\alpha g)u$.
We have $u^{(1)}=u^{(2)}=\cdots = u^{(k-1)}=u$.

Let $h := \l{\rho(\alpha g)}$, and suppose that $h$ has name $c$.
Our arguments will divide into two cases:
(A) $c$ is neither $a$ nor $b$;
(B) $c$ is equal to one of $a$ or $b$.

The following two lemmas summarise the properties that we shall need in these
two cases.

\begin{lemma}\label{wred1} Assume that we are in Case (A). Then: 
\begin{mylist}
\item[(1)] $u^{(k)} = u$,
\item[(2)] If $v^{(n-1)} \neq v^{(n)}$ then 
$u^{(n)}$ involves both $a$ and $b$.
\item[(3)] For each $m$ with $k \le m \le n$, $u^{(k)}$ is equal in $G$ to
a geodesic word having $u^{(m)}$ as a suffix.
\end{mylist}
\end{lemma}
\begin{proof}
(1) is clear from the definition of Case (A).
We examine the reduction of $\rho(\alpha g)u$ to
$\rho(\rho(\alpha g)u)=\rho(v)$.
For each $m>k$, if $v^{(m-1)}$ and
$v^{(m)}$ are distinct, the names of the $m$-th letters of $v^{(m-1)}$ and
$v^{(m)}$ are the two generators of the maximal 2-generator subword that ends at
the $m$-th letter of $v^{(m-1)}$.  Let $l$ be maximal such that $l\geq k$
and the $l$-th letter of $v^{(l)}$ has name $c$.

We see that $l<n$. This is obvious if $l=k$.
If $l>k$ then for each $k < m \leq l$, the prefix of length $m$ in
$v^{(m-1)}$ has a critical suffix involving $c$ and one of $a,b$;  
the fact that $l<n$ follows immediately from the fact that
it must involve the same one each time (for each critical
suffix must end with an alternating subword of length at least 3).

Now by definition of $l$, any reduction of $v^{(m-1)}$
to $v^{(m)}$ with $m>l$ must start with a $\tau$-move involving $a$ and $b$. 
So if $v^{(n-1)} \neq v^{(n)}$  the maximal 2-generator suffices of both
$v^{(n-1)}$ and $v^{(n)}$ must contain both $a$ and $b$, and we have (2).

Now for each $m$ with $k \leq m \leq l$, $u^{(m)}$ is a suffix of $u^{(k)}=u$,
so (3) holds for all such $m$.
For any $m > l$, if $v^{(m)} \ne v^{(m-1)}$,
then the first $\tau$-move in that reduction is to a subword of $u^{(m-1)}$,
and $u^{(m)}$ is a suffix
of the word derived from $u^{(m-1)}$ by applying the first $\tau$-move of
that reduction. Hence we see that we could take the sequence of $\tau$-moves
that form the first steps of each of the non-trivial leftward lex reducing
sequences that reduce $v^{(l)}$ through $v^{(l+1)},v^{(l+2)},\ldots$ to
$v^{(n)}=\rho(v)$,
This sequence of $\tau$-moves  transforms $u$ through a sequence of geodesics
$\hat{u}^{(k+1)},\ldots,\hat{u}^{(n)}$, with $u^{(m)}$ a suffix of
$\hat{u}^{(m)}$ and $\hat{u}^{(m)}=_G u$ and for each $m$ with $k<m \leq n$.
This completes the
proof of (3).
\end{proof}

\begin{lemma}\label{wred2} Assume that we are in Case (B). Then:
\begin{mylist}
\item[(1)] $u^{(k)} = h^ju$ for some $j \ge 1$.
\item[(2)] $u^{(n)}$ involves both $a$ and $b$.
\item[(3)] For each $m$ with $k \le m \le n$, $u^{(k)}$ is equal in $G$ to
a geodesic word having $u^{(m)}$ as a suffix.
\end{mylist}
\end{lemma}
\begin{proof}
It follows from Lemma~\ref{lem2} (2) that
$\rho(\alpha g) = \eta g'h^j$, for some word $\eta$ and $j \ge 1$, where
$g' = g^{\pm 1}$, and so $u^{(k)}=h^ju$, and (1) holds.

To prove (2) and (3), we consider the further reduction of
$v^{(k)} = \rho(\alpha g)u$.
Again we consider the sequence $v^{(k+1)},\cdots,v^{(n)}$ of successive
reductions of $v^{(k)}$ to $v^{(n)}$.
 
We claim that, for any $j'>j$, $\eta g'h^{j'}$ is already reduced.
To see that, note that a critical suffix
$v'$ of $\eta g'h^{j'}$ must have the form $v''h^{j'-j}$ where $v''$ is a
critical suffix of $\eta g'h^j$. And then by Corollary~\ref{critical_subword} 
$\tau(v')$ and $\tau(v'')$ have the same
first letter. So if $v'$ were part of a critical factorisation leading to a
leftward lex reducing sequence of $\eta g'h^{j'}$ then $\eta g'h^j$ would also
have such a reduction, which it does not, since $hg'h^j=\rho(\alpha g) \in W$.

So the first $\tau$-move in any non-trivial reduction of $v^{(m-1)}$ to
$v^{(m)}$ for $k < m \le n$ is to a subword of $u^{(m-1)}$.
Since $u^{(k)}$ involves both $a$ and $b$, the same applies to
$u^{(m)}$ for all $k < m \le n$, which proves (2).

Much as in Case (A), we see that this sequence of first $\tau$-moves
can be applied to $u^{(k)} = h^ju$ to transform it
through a sequence of geodesics
$\hat{u}^{(k+1)},\ldots,\hat{u}^{(n)}$, with $u^{(m)}$ a suffix of
$\hat{u}^{(m)}$ and $\hat{u}^{(m)}=_G u^{(k)}$ and for each $m$ with
$k<m \leq n$, so (3) is true.
\end{proof}

\begin{proposition}\label{twoposs}
Suppose that $v,w$ are any two geodesics in $G$ representing the same group
element, and that $\l{v} \ne \l{w}$.
Then:
\begin{mylist}
\item[(1)] $\l{v}$ and $\l{w}$ have different names;
\item[(2)] The maximal 2-generator suffices of $v$ and $w$ involve generators
with names equal to those of $\l{v}$ and $\l{w}$;
\item[(3)]
Any geodesic word equal in $G$ to $v$ must end in $\l{v}$ or in $\l{w}$.
\end{mylist}
\end{proposition}
\begin{proof} 
Since $\rho(v)=\rho(w)$, 
either $\l{\rho(v)} \neq \l{v}$ or $\l{\rho(w)}\neq \l{w}$. We assume
without loss of generality that $\l{\rho{v}} \neq \l{v}$.
This implies in particular that $v^{(n-1)} \neq v^{(n)}$.

Then $\l{v}=\l{v^{(n-1)}}$ and $\l{\rho(v)}=\l{v^{(n)}}$, and
$v^{(n-1)}$ and $v^{(n)}$ are related by a leftward lex reducing
sequence. Hence we
can deduce from Proposition~\ref{critical_properties}~(3)
that $\l{v}$ and $\l{\rho(v)}$ have distinct names.
If $\l{\rho(v)}=\l{w}$, then it follows immediately that
$\l{v}$ and $\l{w}$ have distinct names, and so (1) holds.
Otherwise  we can repeat
the argument above, replacing $v$ by $w$, to deduce that $\l{w}$ and
$\l{\rho(w)}=\l{\rho(v)}$ have distinct names.
In that case, if (1) is false, then we must have $\l{v}=g$ and $\l{w}=g^{-1}$ for
some $g \in A$,
and so $vg^{-1}$ and $wg$ cannot be geodesic, and neither can 
$\rho(v)g^{-1}$ or $\rho(v)g=\rho(w)g$.
Since both $\rho(v)g^{-1}$ and $\rho(v)g$ are freely reduced, this
contradicts Lemma~\ref{genandinv}. So (1) is true.

Now we prove (2) by induction on $|v|$.
The application of a $\tau$-move to a word does not change the generators 
it involves.
So if $v$ involves at most two generators, then $w$ involves the same
ones, and the result is immediate.

So suppose that $v$ involves at least three generators.
Since $v^{(n-1)} \ne v^{(n)}$, it follows from
Lemmas~\ref{wred1} (2) and~\ref{wred2} (2) that the two generators involved
in the maximal 2-generator suffix of $\rho(v)$ are the same as those
in the maximal 2-generator suffix of $v$.

If $\l{w} \neq \l{\rho(v)}$ then we can apply the argument of the last
paragraph to $w$
in place of $v$, and then (2) is proved. 
So suppose that $\l{w} = \l{\rho(v)}$.
We need to prove that the maximal 2-generator suffix of $w$ involves the same
two generators as that of $\rho(v)$.
Let $v'$ be the result of applying the first $\tau$-move in the reduction
of $v^{(n-1)}$ to $v^{(n)}=\rho(v)$. Then also $\l{w}=\l{v'}$.
Consider the maximal suffix common to $v'$ and $w$.
If this involves two generators then the result is proved, so assume not.
Then $v' = v'_0g^j$ and $w=w_0g^j$ for some $j\geq 1$, and $v'_0=_G w_0$.
Since $v'$ has a critical word as a suffix, $v'_0$ must involve both
of the final two generators  involved in the maximal 2-generator suffix
of $\rho(v)$, so (2) follows by applying induction to the words
$v'_0$ and $w_0$.

(3) now follows from (1) and (2).
\end{proof}

To prove Theorem~\ref{fftp:thm}, it is enough to show that any minimal non-geodesic word
in the generators of $G$ $M$-fellow travels with
a geodesic word representing the same group element. So suppose $vg$ is
minimal non-geodesic with $g \in A$. The result is clear if
$\l{v} = g^{-1}$ so suppose not.  We have $vg =_G v'$ with $|v'| = |v| - 1$ and
hence $w:=v'g^{-1}$ and $v$  are geodesic words representing the same group
element. So it is enough to prove the following proposition.

\begin{proposition}\label{fftp}
Suppose that $v =_G w$ with $v,w$ both geodesic, and
$\l{v} \ne \l{w}$. Then $v$ $M$-fellow travels with a geodesic word $w'$
with $v=_Gw'$ and $\l{w'} = \l{w}$.
\end{proposition}

\begin{proof}
Since $v,w$ are geodesics,
any non-trvial reduction of $v$ to $\rho(v)$, or of $w$ to $\rho(w)=\rho(v)$
must be through a leftward lex reducing sequence.
It follows from Lemma~\ref{lem2} that a leftward lex reducing sequence does
not change the set of generators involved in a word, so $v$ and $w$ involve
the same generators.

The proof is by induction on $n=|v|$.
The base of the induction is provided by the 2-generator result
Corollary~\ref{2gen_fftp}, which also allows us to assume from now on that
$v$ and $w$
involve at least three generators.

Now Proposition~\ref{twoposs}~(2) tells us that the maximal 2-generator
suffices of $v,w$ involve the same two generators. As above we call those
two generators $a,b$, and let $a$ be the name of $\l{v}$. Then the name of
$\l{w}$ is $b$;
it is distinct from the name of $\l{v}$ by Proposition~\ref{twoposs}~(1).

We need to verify the inductive step. So we assume the result holds for
pairs of geodesics of length less than $n$, and verify that it holds for
the given pair of geodesics $v,w$.

We use the following lemma.
\begin{lemma}\label{wred3}
Suppose that $v,w$ satisfy the hypothesis of Proposition~\ref{fftp},
and that the conclusion of Proposition~\ref{fftp} holds for geodesic words
shorter than $v$ and $w$. Then
if $u^{(k)}$ is equal in $G$ to a geodesic word $w_1$ that ends in $\l{w}$, 
the conclusion of  Proposition~\ref{fftp} holds for $v$ and $w$.
\end{lemma}
\begin{proof}
In both Cases (A) and (B), we have $\l{u^{(k)}}=\l{u} = \l{v} \ne \l{w}=\l{w_1}$,
and Corollary~\ref{2gen_fftp} tells us that a single $\tau$-move
can be applied to a suffix of $u^{(k)}$ to transform it to a word ending
in $\l{w}$.

In Case (A), it also follows from Corollary~\ref{2gen_fftp} that 
$u=u^{(k)}$ $M$-fellow travels with a geodesic
word $w_1'$, with $w_1'=_G u$ and $\l{w_1'}=\l{w_1}=\l{w}$.
So $\alpha gw_1'$ $M$-fellow travels with $\alpha gu=v$, and represents the 
same element of $G$ as $v$.

In Case (B), Lemma~\ref{gju_lemma} implies that a single $\tau$-move
can be applied to $hu$ to transform it to a geodesic word $w_2'$
with $\l{w_2'}=\l{w_1}=\l{w}$.
Then $hu$ $M$-fellow travels with $w_2'$ and $w_2'=_G hu$.
Since $\alpha g$ is equal in $G$ to a geodesic word ending in
$h \neq g$ and $|\alpha g| < |v|$, it follows from the hypothesis that
$\alpha g$ $M$-fellow travels with a geodesic word $w_0'$, 
with $\alpha g=_G w_0'$ and $\l{w_0'} = h$.
Now let $w' := \p{w_0'}w_2'$.  Then $\l{w'} = \l{w_2'} = \l{w}$,
\[ w' = \p{w_0'}w_2' =_G \p{w_0'}hu=w_0'u=_G \alpha gu=v, \]
and the fact that $w'$ $M$-fellow travels with $v$ is an immediate
consequence of that fact that the pairs $w_0'$, $\alpha g$ and
$w_2'$, $hu$ $M$-fellow travel.
\end{proof}

If $\l{w} = \l{\rho(v)}$, then we may assume that $w=\rho(v) = v^{(n)}$.
By Lemmas~\ref{wred1}~(3) and~\ref{wred2}~(3), $u^{(k)}$ is equal in $G$ to
a geodesic word ending in $u^{(n)}$ and the conclusion of
Proposition~\ref{fftp} follows immediately from  Lemma~\ref{wred3}.

So we assume from now on that $\l{w} \ne \l{\rho(v)}$.
We can assume (by replacing $w$ by $\rho(\p{w})\l{w}$) that a single
leftward lex reducing sequence transforms $w$ to $\rho(v)=\rho(w)$.

By Proposition~\ref{twoposs} (3), we must have $\l{\rho(v)} = \l{v}$.
We have $\rho(v) = v^{(n)} = \beta g'' u^{(n)}$ for some word $\beta$ and $g''
\in A$, where the name of $g''$ is not equal to $a$ or to $b$.

If the single leftward lex reducing sequence that reduces $w$ to
$\rho(v)$ has length 1, then $w = \beta g'' u'$ with $u' =_G u^{(n)}$.
In that case, by Lemmas~\ref{wred1}~(3) and~\ref{wred2}~(3), $u^{(k)}$ is equal in $G$ to
a geodesic word ending in $\l{w}$, and the result follows by Lemma~\ref{wred3}.

So we suppose that this sequence has length
greater than 1. Then we have $w = \gamma u'$, where $u'$ is the maximal
2-generator suffix of $w$ and $u^{(n)} = \s{\tau(u')}$. 
So $\gamma \f{\tau(u')} =_G \beta g''$.

If $|u^{(n)}| < |u^{(k)}|$, then the $g''$ in $\rho(v)$
appeared as a result of the application of a $\tau$-move during one of the
reductions from $v^{(m-1)}$ to $v^{(m)}$ for some $m>k$. This
application was of the form $v^{(m-1)}=\delta v' u'' \rightarrow \delta \tau(v') u''=v^{(m)}a$, 
where $\l{v'}$ has name $a$ or $b$, $\l{\tau(v')} = g''$,
and $u''$ is a 2-generator suffix of $v^{(m-1)}$.
Since all reductions from $v^{(m'-1)}$ to $v^{(m')}$ for $m'>m$ must consist
of a single $\tau$-move applied to $u^{(m'-1)}$, we have $u''=_G u^{(n)}$.

Then $\delta \tau(v') =_G \delta v'$.
But also \begin{eqnarray*}
 \delta v'u^{(n)}&=_G& \delta v'u''=v^{(m-1)}=_G w = \gamma u'\\
&=_G& \gamma\tau(u') = \gamma\f{\tau(u')}\s{\tau(u')}
=_G \gamma \f{\tau(u')}u^{(n)}\end{eqnarray*} 
so in fact all three of the geodesics $\delta \tau(v')$, $\delta v'$ and
$\gamma \f{\tau(u')}$ represent the same element of $G$.
The first of these has last letter $g''$ (whose name is neither $a$ nor $b$), but the second
and third end in letters with name $a$ or $b$.
So Proposition~\ref{twoposs}~(3) tells us that $\l{v'} = \f{\tau(u')}$,
and hence
\[\l{v'}u''=_G \l{v'}u^{(n)}=_G \f{\tau(u')}\s{\tau(u')} = \tau(u')=_G u'.\]
So a 2-generator suffix of $\delta v'u''$ is equal in $G$
to the 2-generator suffix $u'$ of $w$, which ends in $\l{w}$.
But the maximal 2-generator suffix of $\delta v' u''$ is $u^{(m-1)}$,
and then by
Lemmas~\ref{wred1}~(3) and~\ref{wred2}~(3), $u^{(k)}$ is equal in $G$ to
a word ending in $\l{u^{(n)}}=\l{w}$
and the result follows once again from Lemma~\ref{wred3}.

Otherwise $|u^{(n)}| = |u^{(k)}|$, so
$u^{(n)} =_G u^{(k)}$ and any non-trivial reduction of $v^{(m-1)}$ to $v^{(m)}$ for $m>k$
consists of a single $\tau$-move applied to $u^{(m-1)}$.

In Case (i), we have
$$\alpha g u^{(n)} =_G v =_G w = \gamma u' =_G \gamma \f{\tau(u')} u^{(n)},$$
so $\alpha g =_G \gamma \f{\tau(u')}$. Then, by the inductive hypothesis,
$\alpha  g$ $M$-fellow travels with a word ending in $\f{\tau(u')}$ and
$\f{\tau(u')} u^{(n)} = \tau(u')$ $M$-fellow travels with the word $u'$
ending in $\l{w}$, so the result follows.

Recall that in Case (ii) $\rho(\alpha  g) = \eta g' h^j$.
Since $\rho(v) = \beta g'' u^{(n)} =_G \beta g'' u^{(k)}$, we have
$g'=g''$ and $\beta = \eta$ in this situation.
We saw earlier that $\gamma \f{\tau(u')} =_G \beta g''$, so
$\alpha  g$ and $\gamma \f{\tau(u')} h^j$
are two geodesics representing the same group element.
Since the names of $\f{\tau(u')}$ and $h$ are both $a$ or $b$ and the name of
$g$ is neither $a$ nor $b$, Proposition~\ref{twoposs} (2) implies that $\f{\tau(u')}$
has the same name as $h$ and hence $\f{\tau(u')} = h$.
But now, since
$$h^{j+1} u = \f{\tau(u')} h^j u = \f{\tau(u')} u^{(n)} =_G
\f{\tau(u')}\s{\tau(u')} =_G u',$$
with $\l{u'}=\l{w}$, we can apply
Corollary~\ref{2gen_fftp} and Lemma~\ref{gju_lemma} to deduce that
$h^ju$ $M$-fellow travels with a word
ending in $\l{w}$, and then the result follows from Lemma~\ref{wred3}.
\end{proof}

\end{document}